\documentclass{amsart}
\usepackage{amssymb}
\usepackage{enumerate}
\usepackage{stmaryrd}
\usepackage{tikz}
\usetikzlibrary{arrows,calc,matrix}
\usepackage{hyperref}
\usepackage[msc-links]{amsrefs}
\usepackage[stable]{footmisc}

\newtheorem{theorem}{Theorem}[section]
\newtheorem{lemma}[theorem]{Lemma}
\newtheorem{proposition}[theorem]{Proposition}
\newtheorem*{otherstheorem}{Theorem}

\theoremstyle{remark}
\newtheorem{remark}[theorem]{Remark}

\theoremstyle{definition}
\newtheorem{definition}[theorem]{Definition}

\numberwithin{equation}{section}

\DeclareMathOperator{\ad}{\mathrm{ad}}
\DeclareMathOperator{\Aut}{\mathrm{Aut}}
\DeclareMathOperator{\Atp}{\mathrm{Atp}}
\DeclareMathOperator{\Bij}{\mathrm{Bij}}
\DeclareMathOperator{\Der}{\mathrm{Der}}
\DeclareMathOperator{\Doro}{\mathcal{D}}
\DeclareMathOperator{\Endo}{\mathrm{End}}
\DeclareMathOperator{\Hom}{\mathrm{Hom}}
\DeclareMathOperator{\Id}{\mathrm{Id}}
\DeclareMathOperator{\Lie}{\mathcal{L}}
\DeclareMathOperator{\MH}{\mathcal{MH}}
\DeclareMathOperator{\morco}{\mathrm{Coalg}}
\DeclareMathOperator{\M}{\mathcal{M}}
\DeclareMathOperator{\Mult}{\mathrm{Mlt}}
\DeclareMathOperator{\Nalt}{\mathrm{N}_{\mathrm{alt}}}
\DeclareMathOperator{\PsAut}{\mathrm{PsAut}}
\DeclareMathOperator{\sig}{\mathrm{sig}}
\DeclareMathOperator{\W}{\mathcal{W}}

\newcommand{\C}{\mathcal{C}}
\newcommand{\ep}{\epsilon}
\newcommand{\g}{\mathfrak g}
\newcommand{\Jac}{\mathsf{J}}
\newcommand{\m}{\mathfrak m}
\newcommand{\Oc}{\mathbb O}
\newcommand{\set}[1]{\left\{#1\right\}}
\newcommand{\spann}[1]{\operatorname{span}\langle #1\rangle}
\newcommand{\Sthree}{\mathsf{S}}

\newcommand{\ci}{{c_{(1)}}}
\newcommand{\cii}{{c_{(2)}}}
\newcommand{\ciii}{{c_{(3)}}}
\newcommand{\civ}{{c_{(4)}}}
\newcommand{\cv}{{c_{(5)}}}

\newcommand{\mmi}{{m_{(1)}}}
\newcommand{\mmii}{{m_{(2)}}}
\newcommand{\mmiii}{{m_{(3)}}}
\newcommand{\mmiv}{{m_{(4)}}}
\newcommand{\mmv}{{m_{(5)}}}

\newcommand{\nni}{{n_{(1)}}}
\newcommand{\nnii}{{n_{(2)}}}

\newcommand{\uui}{u_{(1)}}
\newcommand{\uuii}{u_{(2)}}
\newcommand{\uuiii}{u_{(3)}}

\newcommand{\xxi}{{x_{(1)}}}
\newcommand{\xxii}{{x_{(2)}}}
\newcommand{\xxiii}{{x_{(3)}}}
\begin{document}
\title[Hopf algebras with triality]{Hopf algebras with triality}%
\author{Georgia Benkart \and Sara Madariaga \and Jos\'e M. P\'erez--Izquierdo}%
\address{Department of Mathematics, University of Wisconsin, Madison, Wisconsin 53706 USA}
\email{benkart@math.wisc.edu}
\address{Dpto. Matem\'aticas y Computaci\'on, Universidad de La Rioja, 26006, Logro\~no, Espa\~na}
\email{sara.madariaga@unirioja.es}
\address{Dpto. Matem\'aticas y Computaci\'on, Universidad de La Rioja, 26006, Logro\~no, Espa\~na}%
\email{jm.perez@unirioja.es}%
\thanks{Jos\'e M. P\'erez-Izquierdo and Sara Madariaga would like to thank Spanish Ministerio
de Educaci\'on y Ciencia and FEDER MTM 2007-67884-C04-03 and the University of La
Rioja.  Sara Madariaga also thanks support from Spanish MICINN grant
AP2007-01986 and ATUR 09/22.}%

% ----------------------------------------------------------------
\begin{abstract}
In this paper we revisit and extend the constructions of Glauberman and Doro on groups with triality and Moufang loops to Hopf algebras. We prove that the universal enveloping algebra of any Lie algebra with triality is a Hopf algebra with triality. This allows us to give a new construction of the universal enveloping algebras of Malcev algebras. Our work relies on the approach of Grishkov and Zavarnitsine to groups with triality.
\end{abstract}
\maketitle
% ----------------------------------------------------------------
\section{Introduction}

Recall that a \emph{loop} $(Q,\cdot, e)$ is a set with a binary operation $\cdot \colon Q \times Q \to Q$ $(a,b) \mapsto a b$ and a unit element $e \in Q$, i.e. $e a = a = a e$ for any $a \in Q$, such that the multiplication operators $L_a \colon b \mapsto a b$ and $R_b \colon a \mapsto a b$ are bijective for any $a, b \in Q$ \cites{Br58,Pf90}. Roughly speaking a loop is a nonassociative group, or more precisely a group is a loop that in addition satisfies the associative law $(xy) z = x (yz)$.

In the nonassociative setting other loops apart from groups are of interest. One of them is the seven dimensional sphere of octonions of norm $1$. This sphere has no structure of a Lie group, however with the product inherited from the octonions,  it satisfies the (left, middle and right) \emph{Moufang identities}
\begin{displaymath}
a(x(ay)) = ((ax)a)y , \quad (a(xy))a = (ax)(ya) \quad \text{and} \quad ((xa)y)a = x(a(ya))
\end{displaymath}
for any $a, x, y \in Q$, so in some sense this product is nearly associative. Loops that satisfy any of these identities also satisfy the others, and they are called \emph{Moufang loops}.

To any loop $Q$ is attached the group generated by the multiplication operators $\{ L_a , R_ a \mid a \in Q \}$, its \emph{multiplication group} $\Mult(Q)$ and sometimes this group has a strong connection with the structure of $Q$.   This idea of relating loops with groups has been very fruitful for Moufang loops following the work of Glauberman  \cite{Gl68} and Doro \cite{Do78} and especially in recent years \cites{Mi93,GaHa05,GrZa05,GrZa06,GrZa09}. Glauberman observed that the multiplication operators on a Moufang loop $Q$ with identity $1$ satisfy
\begin{equation}
\label{eq:Doro}
    \begin{array}{lll}
      P_1 = L_1 = R _1 = 1, & P_x L_x R_x = 1, &  \\
      L_{xyx} = L_xL_yL_x,  &  R_{xyx} = R_xR_yR_x, & P_{xyx} = P_xP_yP_x, \\
      L_{y^{-1}x} = R_yL_xP_y, & R_{y^{-1}x} = P_yR_xL_y, & P_{y^{-1}x} = L_yP_xR_y, \\
      L_{xy^{-1}} = P_yL_xR_y, & R_{xy^{-1}} = L_yR_xP_y &\text{and }  P_{xy^{-1}} = R_yP_xL_y,
    \end{array}
\end{equation}
where $P_x = R_{x}^{-1}L_{x}^{-1}$. The group $\Doro(Q)$ generated by the symbols $\{ L_x, R_x, P_x \mid x \in Q\}$ subject to relations (\ref{eq:Doro}) inherits two automorphisms $\rho, \sigma$ with $\sigma^2 = \rho^3 = \Id_{\Doro(Q)}$ and $\sigma \rho = \rho^2 \sigma$ such that
\begin{equation}
    \begin{array}{lll}
        P^\rho_x = L_x & L^\rho_x = R_x & R^\rho_x = P_x\\
        P^\sigma_x = P^{-1}_x & L^\sigma_x = R^{-1}_x & R^\sigma_x = L^{-1}_x,
    \end{array}
\end{equation}
due to the symmetries of relations (\ref{eq:Doro}).  They afford a representation of the symmetric group on three letters $\Sthree_3$ as automorphisms of $\Doro(Q)$. One important insight was that
\begin{equation}
\label{eq:triality_Group}
    (g^{-1}g^\sigma)(g^{-1}g^\sigma)^\rho (g^{-1}g^\sigma)^{\rho^2} = 1
\end{equation}
holds for any $g$ in $\Doro(Q)$. Groups $G$ with a representation of $\Sthree_3$ as automorphisms satisfying (\ref{eq:triality_Group}) are called \emph{groups with triality} (relative to $\rho$ and $\sigma$). Surprisingly enough, Doro showed that the construction of a group with triality from a Moufang loop can be reversed.\footnotemark[1]\footnotetext[1]{A detailed and illuminating study of the connections between certain categories of Moufang loops and groups with triality can be found in \cite{Ha1}, where a slightly different definition of a group with triality is used.} We present a simple approach by Grishkov and Zavarnitsine to the construction of a Moufang loop from a group with triality instead of Doro's original approach,  since the former avoids the use of symmetric spaces and cosets.

\begin{otherstheorem}[\cite{GrZa06}]
Given a group with triality $G$, the set $\M(G) = \{ g^{-1}g^\sigma \mid g \in G \}$ is a Moufang loop with respect to the multiplication law
\begin{displaymath}
    m\cdot n = m^{-\rho} n m^{-\rho^2} = n^{-\rho^2} m n^{-\rho} \quad \forall_{m,n \in \M(G).}
\end{displaymath}
\end{otherstheorem}

Any Moufang loop $Q$ is recovered up to isomorphism as $\M(\Doro(Q))$ \cite{GrZa06}. By construction, Doro's group $\Doro(Q)$ satisfies the following universal property: given $G$ a group with triality such that $\M(G) \cong Q$ then there exists a homomorphism of groups with triality $\Doro(Q) \to G$ defined by $P_x \mapsto x, L_x \mapsto x^\rho$ and $R_x \mapsto x^{\rho^2}$ \cites{Do78,GrZa06}.

Mikheev \cite{Mi93} gave another construction of a group with triality $\W(Q)$ with a universal property dual to that of $\Doro(Q)$: if $G$ is a group with triality such that $\M(G) \cong Q$ and $Z_\Sthree(G) = \{ 1_G \}$ then there exists a monomorphism $G \to \W(Q)$ of groups with triality, where $Z_\Sthree(G)$ denotes the maximal normal subgroup of $G$ where $\Sthree = \Sthree_3$ acts trivially. Mikheev's paper has no proofs,  but they were provided by Grishkov and Zavarnitsine in \cite{GrZa06}. The construction of $\W(Q)$ begins with the definition of a \emph{pseudoautomorphism} of a Moufang loop $Q$; that is, a pair $(A,a)$ with $A \colon Q \to Q$ a bijective map and $a$ an element of $Q$, the \emph{right companion of }$A$, related by
\begin{displaymath}
    (xA) \cdot (yA \cdot a) = (x \cdot y)A \cdot a
\end{displaymath}
for all $x,y \in Q$. The pseudoautomorphisms of $Q$ form a group $\PsAut(Q)$ with product
\begin{displaymath}
    (A,a)(B,b) = (AB, aB \cdot b).
\end{displaymath}
The group $\W(Q)$ is then defined as $\W(Q) = \PsAut(Q) \times Q$ with the product
\begin{displaymath}
    [(A,a),x] [(B,b),y] = [(A,a)(B,b)(C,c),xB \cdot y]
\end{displaymath}
where
\begin{displaymath}
    (C,c) = (R_{b,xB},b^{-1}(xB)^{-1}b(xB))(R_{xB,y},(xB)^{-1}y^{-1}(xB)y)
\end{displaymath}
and $R_{x,y} = R_xR_yR^{-1}_{xy}$. The actions of $\rho$ and $\sigma$ are given by
\begin{eqnarray*}
    [(A,a),x] & \stackrel{\rho}{\mapsto} & [(A,a),a][(T_x,x^{-3}),x^{-2}] \text{ and}\cr
    [(A,a),x] & \stackrel{\sigma}{\mapsto} & [(A,a)(T_x,x^{-3}),x^{-1}]
\end{eqnarray*}
with $T_x = L^{-1}_x R_x$.

As remarked in \cite{GrZa06} a direct verification of the associativity of this product is ``technically intractable''. However, this technicality on the definition of $\W(Q)$ is fictitious since this group is just the group of \emph{autotopies} of $Q$, i.e. triples $(A,B,C)$ of bijective maps from $Q$ to $Q$ such that
\begin{displaymath}
    (xy)A = (xB)(yC)
\end{displaymath}
for all $x,y \in Q$, with componentwise product. This group has been known for many years, but its universal property seems to have gone unnoticed until interpreted as $\W(Q)$. This description allows simple proofs of the properties of $\W(Q)$. Section \ref{sec:groups_with_triality} is devoted to this issue. We will also extend these ideas to the context of Hopf algebras in Section \ref{sec:other_constructions}.

The approach of Grishkov and Zavarnitsine to the construction of $\M(G)$ is well suited for its extension to cocommutative Hopf algebras. Recall Sweedler's sigma notation $\Delta(u) = \sum \uui \otimes \uuii$ for the comultiplication in Hopf algebras. The antipode will be usually denoted by $S$, and $\ep$ will stand for the counit. Although the antipode and the group generated by $\rho$ and $\sigma$ are represented by $S$
and $\Sthree$ respectively, this will not lead to confusion. The letter $F$ is reserved for the ground field.

\begin{definition}
Given two automorphisms $\rho, \sigma$ of a cocommutative Hopf algebra $H$ such that $\sigma^2 = \rho^3 = \Id_H$ and $\sigma \rho = \rho^2 \sigma$, $H$ is said to be a cocommutative Hopf algebra with triality relative to $\rho$ and $\sigma$ in case that
\begin{equation}
\label{eq:triality_Hopf}
    \sum P(u_{(1)})\rho(P(u_{(2)}))\rho^2(P(u_{(3)})) = \ep(u)1,
\end{equation}
where $P(u) = \sum \sigma(u_{(1)})S(u_{(2)})$.
\end{definition}

As we will show, the definition of Hopf algebra with triality does not depend on the generators $\rho, \sigma$ of the group $\Sthree=\langle \rho, \sigma\rangle$ generated by $\rho$ and $\sigma$, so we can talk about Hopf algebras with triality $\Sthree$, although usually we will explicitly mention some generators $\rho$ and $\sigma$. The group algebra $FG$ of a group $G$ with triality relative to $\rho$ and $\sigma$ is clearly a cocommutative Hopf algebra with triality relative to $\rho$ and $\sigma$ (we  abuse notation by identifying automorphisms of $G$ with their linear extensions to $FG$) so Hopf algebras with triality are very natural.
Since operators are written on the left in this definition,  $P(g)$ gives $\sigma(g)g^{-1}$ when applied to a group algebra rather than $P(\sigma(g^{-1})) = g^{-1}\sigma(g)$,  which corresponds to the expression $g^{-1}g^{\sigma}$ appearing  in (\ref{eq:triality_Group}).

The analog of Moufang loops in the context of Hopf algebras are Moufang-Hopf algebras:
\begin{definition}
Any cocommutative and coassociative unital bialgebra $(U,\Delta,\ep,\cdot,1)$ satisfying the \emph{left Moufang-Hopf identity}
\begin{equation}
\label{eq:Moufang-Hopf}
    \sum u_{(1)}(v(u_{(2)}w)) = \sum ((u_{(1)}v)u_{(2)})w
\end{equation}
will be called a \emph{Moufang-Hopf algebra} in case that there exists a map $S\colon U\to U$, the antipode, such that
\begin{eqnarray*}
    && \sum S(\uui)(\uuii v) = \ep(u)v = \sum \uui (S(\uuii)v) \text{ and }\\
    && \sum (v\uui)S(\uuii) = \ep(u)v = \sum (v S(\uui))\uuii.
\end{eqnarray*}
\end{definition}
Any Moufang-Hopf algebra also satisfies the \emph{middle} and \emph{right Moufang-Hopf identities}:
\begin{displaymath}
    \sum (\uui (vw)) \uuii = \sum (\uui v)(w \uuii) \text{ and } \sum ((v \uui)w)\uuii = \sum v(\uui (w\uuii)).
\end{displaymath}

The \emph{loop algebra} $FQ$ of a Moufang loop $Q$ is an example of Moufang-Hopf algebra (with $\Delta(a) = a \otimes a$, $\ep(a) = 1$,  and antipode that  is the linear extension of $S\colon a \mapsto a^{-1}$ for any $a\in Q$). Other sources of Moufang-Hopf algebras are the universal enveloping algebras of Malcev algebras. A \emph{Malcev} algebra over a field of characteristic $\neq 2$ is an algebra $(\m,[\,,\,])$ with a skew-symmetric product $[x,y]$ that satisfies the Malcev identity
\begin{displaymath}
   \Jac(x,y,[x,z]) = [\Jac(x,y,z),x]
\end{displaymath}
where $\Jac(x,y,z) = [[x,y],z] + [[y,z],x] + [[z,x],y]$. In the same way that any associative algebra becomes a Lie algebra with the commutator product, for any nonassociative algebra $A$ the \emph{generalized alternative nucleus}
\begin{displaymath}
    \Nalt(A) = \{ a \in A \mid (a,x,y) = - (x,a,y) = (x,y,a) \quad \forall_{x,y \in A} \}
\end{displaymath}
where $(x,y,z) = (xy)z - x(yz)$, is closed under the commutator $[x,y] = xy - yx$ and with this product becomes a Malcev algebra. Any Lie algebra $\g$ appears as a (Lie) subalgebra of its universal enveloping algebra $U(\g)$ when the associative product is replaced by the commutator. For Malcev algebras there is a counterpart of this result: any Malcev algebra $\m$ over a field of characteristic $\neq 2,3$ appears as a (Malcev) subalgebra of $\Nalt(U(\m))$ for some nonassociative algebra $U(\m)$,  namely its universal enveloping algebra \cite{PeSh04}. Whenever the Malcev algebra $\m$ is a Lie algebra, $U(\m)$ is isomorphic to the usual universal enveloping algebra of the Lie algebra $\m$. This result was the aim of \cite{PeSh04},  where it was noticed that $U(\m)$ also has a bialgebra structure. Later, in \cite{Pe07} it was proved that $U(\m)$ is a Moufang-Hopf algebra where $\m$ embeds as primitive elements, i.e. elements $a$ such that $\Delta(a) = a \otimes 1 + 1 \otimes a$
 , where $\Delta$ stands for the comultiplication. In any Moufang-Hopf algebra, the Moufang-Hopf identities easily imply that  primitive elements belong to the generalized alternative nucleus.

The constructions  of Glauberman, Doro, Grishkov and Zavarnitsine can be extended to cocommutative Hopf algebras with triality in the following terms:
\begin{otherstheorem}
Let $H$ be a cocommutative Hopf algebra with triality relative to $\rho$ and $\sigma$ and define $P(x) =\sum \sigma(\xxi) S(\xxii)$ for any $x \in H$. Then
\begin{displaymath}
\MH(H) = \{ P(x) \mid x \in H\}
\end{displaymath}
is a unital cocommutative Moufang-Hopf algebra with the coalgebra structure and antipode inherited from $H$, the same unit element,  and product defined by
\begin{displaymath}
    u*v = \sum \rho^2(S(\uui)) v \rho(S(\uuii)) = \sum \rho(S(v_{(1)})) u \rho^2(S(v_{(2)}))
\end{displaymath}
for any $u,v \in \MH(H)$.
\end{otherstheorem}
We will devote Section \ref{sec:Hopf_algebras_with_triality} to proving this result. Doro's construction of $\Doro(Q)$ can also be extended to the context of Hopf algebras to obtain a converse of this theorem, namely that any cocommutative Moufang-Hopf algebra appears as $\MH(H)$ for a certain  Hopf algebra $H$ with triality.  In particular, there should be a natural way of constructing the universal enveloping algebra of a Malcev algebra from a Hopf algebra with triality. The development of this approach was the motivation for the present paper.

We need a final ingredient to put all the pieces together, namely, the notion of a Lie algebra with triality that appeared in the work of Mikheev \cite{Mi92} and was studied by Grishkov in \cite{Gr03}. Given a Lie algebra $\g$, two automorphisms $\rho, \sigma$ of  $\g$ such that $\sigma^2 = \rho^3 = \Id_\g$, $\sigma \rho = \rho^2\sigma$ and $\Sthree = \langle \rho,\sigma\rangle$,  the group generated by them, $\g$ is said to be a Lie algebra with triality $
\Sthree$ (or relative to $\rho$ and $\sigma$) in case that
\begin{equation}
\label{eq:triality_Lie}
a - \sigma(a) + \rho(a) -\rho\sigma(a) +\rho^2(a) - \rho^2\sigma(a) = 0
\end{equation}
for any $a\in \g$. The automorphisms $\rho$, $\sigma$ induce an action $\lambda \colon \Sthree_3 \to \Aut(\g)$ of the symmetric group on three letters $\Sthree_3$ on $\g$ by $(12)\mapsto \sigma$, $(123)\mapsto \rho$. Condition (\ref{eq:triality_Lie}) is equivalent to
\begin{equation}
\label{eq:alternate_sum}
\sum_{\tau\in \Sthree_3} \sig(\tau) \tau(a) = 0,
\end{equation}
where we write $\tau(a)$ instead of $\lambda(\tau)(a)$ for short. In particular (\ref{eq:triality_Lie}) does not depend on the choice of the generators $\rho, \sigma$ of $\Sthree$.

Section \ref{sec:Lie_algebras_with_triality} is devoted to proving  that the universal enveloping algebra of a Lie algebra with triality is a Hopf algebra with triality, so we can produce Moufang-Hopf algebras from Lie algebras with triality.

In Section \ref{sec:universal_enveloping_algebra} we will  present a different approach to the construction of the universal enveloping algebra $U(\m)$ of a Malcev algebra $\m$ in \cite{PeSh04}. We start with a Malcev algebra $\m$ and the Lie algebra $\Lie(\m)$ defined in \cite{PeSh04}. This Lie algebra happens to be a Lie algebra with triality, so
 its universal enveloping algebra $U(\Lie(\m))$ is a Hopf algebra with triality. We will prove that $U(\m)$ is isomorphic to $\MH(U(\Lie(\m)))$. Figure \ref{fig:diagram} shows some relations between the objects that we will be concerned with.

\begin{figure}
\label{fig:diagram}
 \begin{tikzpicture}
  \node(Hopf) at (0,0)  [draw, style = very thick, rounded corners] {Hopf algebras with triality};
  \node(Lie) at (0,3)  [draw] {Lie algebras with triality};
  \node(Malcev) at (7,3)  [draw] {Malcev algebras};
  \node(MoufangHopf) at (7,0)  [draw] {Moufang-Hopf algebras};
  \node(Group) at (0,-3)  [draw] {Groups with triality};
  \node(Moufang) at (7,-3)  [draw] {Moufang loops};

\draw[->,very thick] ($(Lie.south) + (-.4em,0)$) -- node[left]{ $\mathfrak{g} \mapsto U(\mathfrak{g})$} ($(Hopf.north) + (-.4em,0)$);
\draw[<-,dotted] ($(Lie.south) + (.4em,0)$) -- node[right]{primitive elements} ($(Hopf.north) + (.4em,0)$);
\draw[->,thick] ($(Malcev.south) + (.4em,0)$) -- node[right]{$\mathfrak{m} \mapsto U(\mathfrak{m})$}  ($(MoufangHopf.north) + (.4em,0)$);
\draw[<-,dotted] ($(Malcev.south) + (-.4em,0)$) -- node[left]{primitive elements} ($(MoufangHopf.north) + (-.4em,0)$);
\draw[->,thick] ($(Malcev.west) + (0,.4em)$) -- node[above]{$\Lie(\m) \mapsfrom \m$} ($(Lie.east) + (0,.4em)$) ;
\draw[<-,dotted] ($(Malcev.west) + (0,-.4em)$) -- ($(Lie.east) + (0,-.4em)$) ;
\draw[->,very thick] ($(MoufangHopf.west) + (0,-.4em)$) -- node[below]{$\Doro(U) \mapsfrom U$} ($(Hopf.east) + (0,-.4em)$) ;
\draw[<-,very thick] ($(MoufangHopf.west) + (0,.4em)$) -- node[above]{$H \mapsto \MH(H)$} ($(Hopf.east) + (0,.4em)$) ;
\draw[->,thick] ($(Moufang.west) + (0,-.4em)$) -- node[below]{$\Doro(Q) \mapsfrom Q$} ($(Group.east) + (0,-.4em)$) ;
\draw[<-,dotted] ($(Moufang.west) + (0,.4em)$) -- node[above]{$G \mapsto \M(G)$}($(Group.east) + (0,.4em)$) ;
\draw[<-,thick] ($(Hopf.south) + (-.4em,0)$) -- node[left]{ group algebra} ($(Group.north) + (-.4em,0)$);
\draw[->,dotted] ($(Hopf.south) + (.4em,0)$) -- node[right]{group-like elements} ($(Group.north) + (.4em,0)$);
\draw[<-,thick] ($(MoufangHopf.south) + (.4em,0)$) -- node[right]{loop algebra} ($(Moufang.north) + (.4em,0)$);
\draw[->,dotted] ($(MoufangHopf.south) + (-.4em,0)$) -- node[left]{group-like elements} ($(Moufang.north) + (-.4em,0)$);
\end{tikzpicture}
\caption{Some relations between Moufang/Malcev objects and associative/Lie objects with triality}
\end{figure}
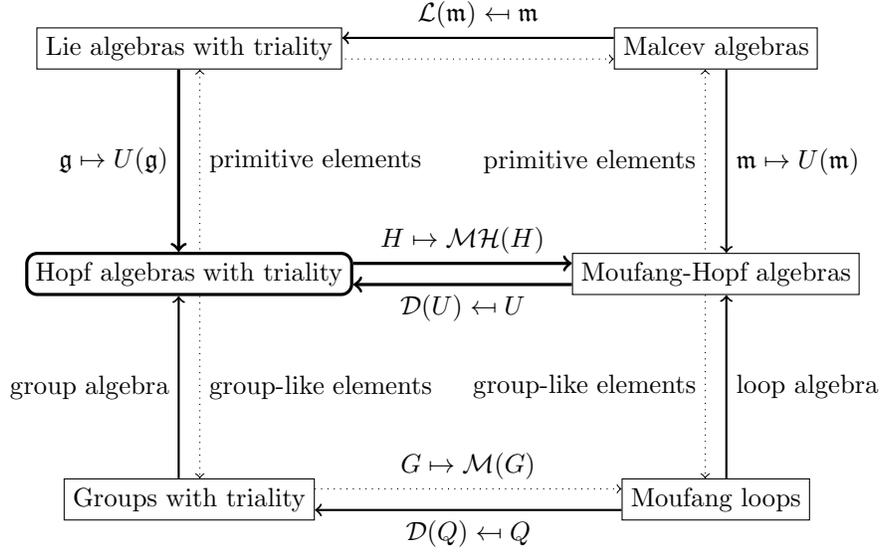

%%%%%%%%%%%%%%%%%%%%%%%%%%%%%%%%%%%%%%
%%%%%%%%%%%%%%%%%%%%%%%%%%%%%%%%%%%%%%
%              %
%             %%
%            % %
%              %
%              %
%              %
%             %%%
%%%%%%%%%%%%%%%%%%%%%%%%%%%%%%%%%%%%%%
%%%%%%%%%%%%%%%%%%%%%%%%%%%%%%%%%%%%%%

\section{Groups with triality and Moufang loops\footnotemark[2]}

\label{sec:groups_with_triality}

\footnotetext[2]{After the first version of this paper was submitted to the editor, Jonathan Hall called our attention to his recent work \cite{Ha10} devoted to proving Theorem \ref{thm:Hall} below for Bol loops, and he provided us with a copy of his preprint \cite{Ha1}. This has led us to remove our proof of Theorem \ref{thm:Hall}  from the present version of our paper. Theorem \ref{thm:GrZa} has been independently established in \cite{Ha1}*{Section 10.3}. However, our approach to this result is the motivation for its generalization to Moufang-Hopf algebras in Section \ref{sec:other_constructions}, and so is included here.  We take this opportunity to express our gratitude to Professor Hall for making his work available to us. }

Given a Moufang loop $Q$, an \emph{autotopy} of $Q$ is a triple $(A_1,A_2,A_3)$ of bijective transformations in $\Bij(Q)$ such that
\begin{displaymath}
    (xy)A_1 = (xA_2)(yA_3)
\end{displaymath}
for any $x,y \in Q$. The set $\Atp(Q)$ of all autotopies of $Q$ is a group with the componentwise composition. The Moufang identities imply that
\begin{displaymath}
    (L_x,U_x,L^{-1}_x), (R_x, R^{-1}_x, U_x) \text{ and } (U_x, L_x, R_x)
\end{displaymath}
are autotopies of $Q$, where $U_x = L_x R_x$. There is an action of the symmetric group on three letters as automorphisms of $\Atp(Q)$ given by
\begin{equation}
\label{eq:rho_sigma_autotopy}
   (A_1, A_2, A_3)^\rho = (JA_2J, A_3, JA_1J)  \text{ and } (A_1, A_2, A_3)^\sigma = (A_3, JA_2J,A_1)
\end{equation}
where $J\colon x \mapsto x^{-1}$ for any $x \in Q$ \cite{Sm99}*{Proposition 4.1.1}. In addition to the identities that define a Moufang loop,  it is convenient to recall that  any element $x$ in the loop has an  inverse $x^{-1}$ which satisfies $L^{-1}_x = L_{x^{-1}}$ and  $R^{-1}_x = R_{x^{-1}}$.

\begin{lemma}
\label{lem:middle}
If $(A_1,A_2,A_3) \in \Atp(Q)$ with $1 A_2 = 1$, then $A_1 = A_3$ and $JA_2J = A_2$.
\end{lemma}
\begin{proof}
The condition on $A_2$ implies that $(yA_1) = (1y)A_1 = (1A_2)(yA_3) = yA_3$ so $A_1 = A_3$. With $a = 1 A_3$ we obtain that $x A_3 = x A_1 = (x1)A_1 = xA_2 a$, so $a = 1A_1 = (xx^{-1})A_1 = (xA_2)(x^{-1}A_3) = (xA_2)(x^{-1}A_2 a)$. Multiplying by $(xA_2)^{-1}$ we get that $(xA_2)^{-1}a = x^{-1}A_2 a$ which gives the result.
\end{proof}

Compare the following theorem with \cite{GrZa06}*{Corollary 1}.

\begin{theorem}
\label{thm:GrZa}
Let $Q$ be a Moufang loop. Then $\Atp(Q)$ is a group with triality (relative to $\rho$ and $\sigma$ given by (\ref{eq:rho_sigma_autotopy})) such that $\M(\Atp(Q)) \cong Q$ and $Z_\Sthree(\Atp(Q)) = \{ 1_{\Atp(Q)} \}$. Moreover, $\Atp(Q)$ is a universal injective object in the following sense: if $G$ is any group with triality such that $\M(G) \cong Q$ and $Z_\Sthree(G) = \{ 1_G \}$, then there exists a monomorphism of groups with triality $G \to \Atp(Q)$.
\end{theorem}
\begin{proof}
The condition of being a group with triality for $\Atp(Q)$ is equivalent to the following three equalities:
\begin{eqnarray*}
    A_1^{-1}A_3JA^{-1}_2JA_2JA^{-1}_3A_1 J &=& \Id_Q, \\
    A^{-1}_2JA_2JA^{-1}_3A_1JA^{-1}_1A_3J &=& \Id_Q \text{ and }\\
    A^{-1}_3A_1JA^{-1}_1A_3JA^{-1}_2JA_2J &=& \Id_Q.
\end{eqnarray*}
To check these equalities, we first observe that by choosing $x = 1A_2$, the middle component of $(A_1,A_2,A_3)(R_x,R^{-1}_x,U_x) \in \Atp(Q)$ fixes $1$ so, according to Lemma~\ref{lem:middle},  $(A_1,A_2,A_3)(R_x,R^{-1}_x,U_x) = (A,B,A)$ for some $A,B$ with $JB = BJ$. Since
\begin{displaymath}
    (A_1,A_2,A_3) = (A, B, A) (R^{-1}_x,R_x,U^{-1}_x)
\end{displaymath}
and $JL_xJ = R^{-1}_x$ then
\begin{eqnarray*}
    A_1^{-1}A_3JA^{-1}_2JA_2JA^{-1}_3A_1 J &=& R_xA^{-1}AU^{-1}_xJR^{-1}_xB^{-1}JBR_xJU_xA^{-1}AR^{-1}_xJ \\
    &=& L^{-1}_x JR^{-1}_xJ R_xJL_xJ = \Id_Q.
\end{eqnarray*}
The other identities are proved in a similar way.

The set $\M(\Atp(Q))$ consists of autotopies of the form
\begin{equation}
\label{eq:aux1}
    (A^{-1}_1A_3, A^{-1}_2JA_2J,A^{-1}_3A_1).
\end{equation}
Since in any autotopy $(A_1,A_2,A_3)$, $(xy)A_1 = (xA_2)(yA_3)$ implies that $A_1 = A_3L_a$ with $a = 1A_2$, then $A^{-1}_1A_3 = L^{-1}_a$ and the autotopy (\ref{eq:aux1}) can be written as $(L^{-1}_a,A^{-1}_2JA_2J,L_a)$. Again, in any autotopy $(A_1,A_2,A_3)$ the first and second components are related by $A_1 = A_2R_{1A_3}$, so $(L^{-1}_a,A^{-1}_2JA_2J,L_a) = (L^{-1}_a,U^{-1}_a,L_a)$. Since any element $a$ in $Q$ is of the form $1A_2$ for a certain autotopy $(A_1,A_2,A_3)$, then
\begin{displaymath}
    \M(\Atp(Q)) = \{ (L^{-1}_a,U^{-1}_a,L_a) \mid a \in Q \}.
\end{displaymath}
The product of $(L^{-1}_a, U^{-1}_a, L_a)$ and $(L^{-1}_b, U^{-1}_b, L_b)$ in $\M(\Atp(Q))$ is
\begin{eqnarray*}
    (L^{-1}_a, U^{-1}_a, L_a)\cdot (L^{-1}_b, U^{-1}_b, L_b) &=& (U^{-1}_a, L^{-1}_a, R^{-1}_a) (L^{-1}_b, U^{-1}_b, L_b) (R_a, R^{-1}_a, U_a) \\
    &=& (L^{-1}_{ab}, U^{-1}_{ab}, L_{ab})
\end{eqnarray*}
where the last equality follows from the Moufang identities. This proves that $(L^{-1}_b, U^{-1}_b, L_b) \mapsto b$ gives an isomorphism $\M(\Atp(Q)) \cong Q$.

In any group $G$ with triality,  $z \in Z_\Sthree(G)$ if and only if $z^\tau = z$ and $g^{-1}g^\tau z = z g^{-1}g^\tau$ for any $g \in G$ and $\tau \in \Sthree$, so any $(A_1, A_2, A_3) \in Z_\Sthree(\Atp(Q))$ i) satisfies $A_1 = A_2 = A_3$ and  ii) commutes with $(L^{-1}_a, U^{-1}_a,L_a)$ for all $a \in Q$. Condition i) implies that $A_1$ is an automorphism of $Q$, and condition ii) then says that $A_1 = \Id_Q$. Thus, $Z_\Sthree(\Atp(Q)) = \{ 1_{\Atp(Q)} \}$.

Let $G$ be a group with triality relative to $\rho$ and $\sigma$ with $\M(G) \cong Q$ (there will be no confusion in using the same letters to denote the automorphisms of $G$ and $\Atp(Q)$). Since
\begin{displaymath}
    x^{-1}(g^{-1}g^\sigma)x^\sigma = (gx)^{-1} (gx)^\sigma \in \M(G)
\end{displaymath}
for any $x, g \in G$,  we can define maps $A_1, A_2, A_3 \colon G \to \Bij(\M(G))$ by
\begin{eqnarray*}
    A_1 \colon x &\mapsto& A_1(x) \colon m \mapsto x^{-\rho^2\sigma} m x^{\rho^2},\\
    A_2 \colon x &\mapsto & A_2(x) \colon m \mapsto x^{-1} m x^\sigma \text{ and }\\
    A_3 \colon x &\mapsto & A_3(x) \colon m \mapsto x^{-\rho}m x^{\rho \sigma}.
\end{eqnarray*}
Clearly
\begin{eqnarray*}
    (m \cdot n)^{A_1(x)} &=& x^{-\rho^2 \sigma} m ^{-\rho} n m^{-\rho^2} x^{\rho^2}\\
    &=& x^{-\sigma \rho} m^{-\rho} x^\rho x^{-\rho} n x^{\rho \sigma} x^{-\sigma \rho^2} m^{-\rho^2} x^{\rho^2}\\
    &=& (x^{-1} m x^\sigma) \cdot (x^{-\rho} n x^{\rho \sigma}) \\
    &=& m^{A_2(x)} \cdot n^{A_3(x)}
\end{eqnarray*}
and $m^{A_1(xy)} = (xy)^{-\rho^2\sigma} m (xy)^{\rho^2} = (m^{A_1(xy)})^{A_1(y)}$ so $A_1(xy) = A_1(x) A_1(y)$. In the same way $A_2(xy) = A_2(x)A_2(y)$ and $A_3(xy) = A_3(x)A_3(y)$. Therefore, we get a homomorphism of groups
\begin{eqnarray*}
 G & \to & \Atp(\M(G)) \\
 x & \mapsto& (A_1(x), A_2(x), A_3(x)).
\end{eqnarray*}
One can easily check that this homomorphism  is a homomorphism of groups with triality. The kernel of this homomorphism consists of the elements $x \in G$ such that $x^{-1} m x^\sigma = m = x^{-\rho} m x^{\rho \sigma}$ for any $m \in \M(G)$. This is equivalent to $x^\sigma = x = x^\rho$ and $xm = mx$ for all $m \in \M(G)$, so the kernel of the previous homomorphism is $Z_\Sthree(G)$.
\end{proof}

To conclude this section, we would like to describe an explicit isomorphism between $\W(Q)$ and $\Atp(Q)$.  First we interpret pseudoautomorphisms in terms of autotopies.

\begin{lemma}
\label{lem:pseudoautomorphisms}
Let $Q$ be an arbitrary loop. Then the map
\begin{eqnarray*}
    \{(A_1,A_2,A_3) \in \Atp(Q) \mid 1A_2 = 1\} & \mapsto & \PsAut(Q)\\
    (A_1,A_2,A_3) & \mapsto & (A_2, 1A_1)
\end{eqnarray*}
is a group isomorphism.
\end{lemma}

As noted in the proof of Theorem \ref{thm:GrZa}, any autotopy $(A_1, A_2, A_3)$ of a Moufang loop decomposes in a unique way in $\Atp(Q)$ as
\begin{displaymath}
(A_1, A_2, A_3) = (A',A,A') (R^{-1}_x,R_x,U^{-1}_x)
\end{displaymath}
with $x = 1 A_2$ and certain $A, A'$. Hence, we may identify $\Atp(Q)$ with $\PsAut(Q) \times Q$ through $(A_1, A_2, A_3) \mapsto [(A,a),x]$ with $a = 1A'$ and  $x = 1 A_2$.

\begin{theorem}
\label{thm:Hall}
Let $Q$ be a Moufang loop. The map
\begin{eqnarray*}
    \psi \colon \Atp(Q) &\to& \W(Q) \\
    (A_1, A_2, A_3) & \mapsto & [(A,a),x]
\end{eqnarray*}
with $x = 1A_2$, $A = A_2R^{-1}_x$ and $ a = 1A_1 x$ is an isomorphism of groups with triality.
\end{theorem}

The reader may have noticed that  in this section most of the operators act  on the right of their arguments, which we have done to be consistent with \cite{GrZa06}. We will return to this notation in Section \ref{sec:other_constructions} where we will consider an analog of autotopies for Moufang-Hopf algebras. In the intervening  sections,  operators will act on the left of their arguments.

%%%%%%%%%%%%%%%%%%%%%%%%%%%%%%%%%%%%%%
%%%%%%%%%%%%%%%%%%%%%%%%%%%%%%%%%%%%%%
%              %
%             % %
%               %
%              %
%             %
%            %
%            %%%%
%%%%%%%%%%%%%%%%%%%%%%%%%%%%%%%%%%%%%%
%%%%%%%%%%%%%%%%%%%%%%%%%%%%%%%%%%%%%%

\section{Cocommutative Hopf algebras with triality}
\label{sec:Hopf_algebras_with_triality}
In this section we will prove that any cocommutative Hopf algebra $H$ with triality relative to $\rho$  and $\sigma$ induces a Moufang-Hopf algebra. The arguments are the natural extension of those in \cite{GrZa06} to Hopf algebras. We define
\begin{displaymath}
\MH(H)=\set{P(x)\mid x\in H}
\end{displaymath}
where $P(x) = \sum \sigma(\xxi) S(\xxii)$. Notice that $S(P(x)) = \sigma(P(x))= P(\sigma(x))$ and that $\Delta(\MH(H))\subseteq \MH(H) \otimes \MH(H)$.

\begin{lemma}
\label{lem:commutation}
For any $u,v\in \MH(H)$ the following hold
\begin{enumerate}[\rm a)]
    \item $\sum \rho^i(\uui) \rho^j(\uuii) = \sum \rho^j(\uui)\rho^i(\uuii)$ ($i,j\in\{0,1,2\}$) and
    \item $\sum \rho^2(S(u_{(1)}))v\rho(S(u_{(2)}))=\sum   \rho(S(v_{(1)}))u\rho^2(S(v_{(2)})) \in \MH(H)$.
\end{enumerate}
\end{lemma}
\begin{proof}
On the one hand, condition (\ref{eq:triality_Hopf}) implies that $\sum \uui \rho(\uuii)\rho^2(\uuiii) = \ep(u) 1$ hence $\sum \uui \rho(\uuii) = S(\rho^2(u))$. On the other hand, equation (\ref{eq:triality_Hopf}) applied to $S(u)$ gives $\sum S(\uui) \rho(S(\uuii))\rho^2(S(\uuiii)) = \ep(u) 1$ so $\sum \rho^2(\uui)\rho(\uuii)\uuiii = \ep(u)1$ from which $\sum \rho(\uui)\uuii = S(\rho^2(u)) = \sum \uui \rho(\uuii)$. This proves  part a) in case that $i =1, j=0$. The other cases follow from this one by using $\rho$.

Since $\sigma(u) = S(u)$ and $\sigma(v) = S(v)$, then
\begin{eqnarray*}
    P\left( \rho(u)\rho(S(v)) \right) & = & \sum \sigma(\rho(u_{(1)})\rho(S(v_{(1)})))S(\rho(u_{(2)})\rho(S(v_{(2)})))\\
    & = & \sum \rho^2 (S(u_{(1)})) \rho^2(v_{(1)}) \rho(v_{(2)}) S(\rho(u_{(2)}))\\
    & = & \sum \rho^2(S(u_{(1)})) v \rho(S(u_{(2)}))
\end{eqnarray*}
and we obtain that $\sum \rho^2(S(\uui))v\rho(S(\uuii))\in \MH(H)$. In the same way the element $\sum \rho(S(v_{(1)})) u \rho^2(S(v_{(2)}))$ belongs to $\MH(H)$. We can use the triality condition on this element to obtain that
\begin{eqnarray*}
   \ep(uv) 1 &=& \sum \rho(S(v_{(1)})) \uui \rho^2(S(v_{(2)})) S(v_{(3)}) \rho^2(\uuii) \rho(S(v_{(4)})) \\
   &&\quad\quad \rho^2(S(v_{(5)})) \rho(\uuiii) S(v_{(6)})\\
   &=& \sum \rho(S(v_{(1)})) \uui \rho(v_{(2)}) \rho^2(\uuii) v_{(3)} \rho(\uuiii) S(v_{(4)})
\end{eqnarray*}
so
\begin{eqnarray*}
   \ep(uv) 1 &=& \sum S(v_{(1)})\rho(S(v_{(2)})) \uui \rho(v_{(3)}) \rho^2(\uuii) v_{(4)} \rho(\uuiii) \\
   &=& \sum \rho^2(v_{(1)}) \uui \rho(v_{(2)}) \rho^2(\uuii) v_{(3)} \rho(\uuiii).
\end{eqnarray*}
This equation implies that $\sum \rho^2(u_{(1)}) v \rho(u_{(2)}) =  \sum  \rho(S(v_{(1)}))S(u)\rho^2(S(v_{(2)}))$.
Replacing $u$ by $S(u)$ we get part b).
\end{proof}

\begin{remark}
Part a) of the lemma shows that the role of $\rho$ and $\rho^2$ can be switched in the definition of Hopf algebra with triality. Moreover, assume that $\rho\sigma$ is chosen instead of $\sigma$ and set $Q(x) = \sum \rho\sigma(x_{(1)}) S(x_{(2)})$. Then $Q(x) = \rho^2(P(\rho(x)))$ so the triality relation $\sum Q(x_{(1)})\rho(Q(x_{(2)}))\rho^2(Q(x_{(3)})) =\ep(x)1$ holds if and only if $\sum P(x_{(1)})\rho(P(x_{(2)}))\rho^2(P(x_{(3)})) = \ep(x) 1$. Therefore, $\sigma$ can be replaced by $\rho\sigma$ in the definition of Hopf algebra with triality. The same is true for $\rho^2\sigma$, so the definition does not depend on the generators $\rho, \sigma$ of the group $\Sthree=\langle \rho,\sigma \rangle$.
\end{remark}

\begin{theorem}
\label{thm:product_MH}
Let $H$ be a cocommutative Hopf algebra with triality relative to $\rho$ and $\sigma$. Then $\MH(H)$ is a unital cocommutative Moufang-Hopf algebra with the coalgebra structure and antipode inherited from $H$, the same unit element and product defined by
\begin{displaymath}
    u*v = \sum \rho^2(S(\uui)) v \rho(S(\uuii)) = \sum \rho(S(v_{(1)})) u \rho^2(S(v_{(2)})).
\end{displaymath}
\end{theorem}
\begin{proof}
We first observe that since $\Delta(P(x)) = \sum P(\xxi)\otimes P(\xxii)$, then $\MH(H)$ is a subcoalgebra of $H$.
The product is well defined because of Lemma \ref{lem:commutation}. This product is clearly a homomorphism of coalgebras $ \MH(H)\otimes \MH(H) \rightarrow \MH(H)$ so $\MH(H)$ is a (nonassociative) bialgebra. The unit element of $H$ satisfies $1*v = 1v1 = v$ and $u*1 = 1u1 = u$ by Lemma \ref{lem:commutation}. The antipode $S$ restricts to a corresponding antipode on $\MH(H)$. In fact,
\begin{eqnarray*}
    \sum S(\uui)*(\uuii*v) &=& \sum S(\uui)*(\rho^2(S(\uuii)) v \rho(S(\uuiii))) \\
    &=& \sum \rho^2(u_{(1)})\rho^2(S(u_{(2)})) v \rho(S(u_{(3)})) \rho(u_{(4)})\\
    &=& \ep(u)v\\
    &=& \sum \uui*(S(\uuii)*v).
\end{eqnarray*}
Since $S(u*v) = \sum \rho(u_{(1)}) S(v) \rho^2(u_{(2)}) = S(v)*S(u)$, then we get
$$\sum (v*\uui)*S(\uuii) = \ep(u)v = \sum (v*S(\uui))*\uuii.$$

Finally,
\begin{eqnarray*}
 && \hskip -.8 truein  \sum ((\uui * v)*\uuii)*w =  \sum ((\rho^2(S(\uui))v\rho(S(\uuii)))*\uuiii)*w\\
    &\qquad  \quad =&  \sum (\uui \rho^2(S(v_{(1)})) \rho(\uuii) \uuiii \rho^2(u_{(4)})  \rho(S(v_{(2)})) u_{(5)})* w \\
    &\qquad \quad =& \sum (\uui \rho^2(S(v_{(1)})) \rho(S(v_{(2)})) u_{(2)})* w\\
    &\qquad \quad =& \sum (\uui v u_{(2)})* w\\
    &\qquad \quad =& \sum \rho^2(S(u_{(1)}))\rho^2(S(v_{(1)}))\rho^2(S(u_{(2)}))w \rho(S(u_{(3)}))\rho(S(v_{(2)}))\rho(S(u_{(4)}))  \\
    &\qquad \quad =& \sum u_{(1)}*(\rho^2(S(v_{(1)}))\rho^2(S(u_{(2)}))w \rho(S(u_{(3)}))\rho(S(v_{(2)})))  \\
    &\qquad \quad =& \sum u_{(1)}*(v*(\rho^2(S(u_{(2)}))w \rho(S(u_{(3)})))  \\
    &\qquad \quad =& \sum u_{(1)}*(v*(u_{(2)}*w)).
\end{eqnarray*}
\end{proof}

For any Moufang-Hopf algebra $U$ and $m \in U$ consider $P_m = \sum R_{S(\mmi)}L_{S(\mmii)}$.

\begin{lemma}
\label{lem:multiplication_algebra}
Let $U$ be a cocommutative Moufang-Hopf algebra. Then for any $m, n \in U$:
\begin{itemize}
    \item[i)] $P_1 = L_1 = R_1 = \Id_U$,
    \item[ii)] $\sum P_\mmi L_\mmii R_\mmiii = \ep(m) \Id_U$,
    \item[iii)] $\sum P_\mmi P_n P_\mmii = \sum P_{\mmi n \mmii}$, \ \ $ \sum L_\mmi L_n L_\mmii = \sum L_{\mmi n \mmii}$, \newline $ \sum R_\mmi R_n R_\mmii = \sum R_{\mmi n \mmii}$,
    \item[iv)] $\sum  R_\mmi P_n L_\mmii = P_{S(m)n}$, \ \ $\sum  P_\mmi L_n R_\mmii = L_{S(m)n}$, \newline $\sum  L_\mmi R_n P_\mmii = R_{S(m)n}$,
    \item[v)] $\sum  L_\mmi P_n R_\mmii = P_{nS(m)}$, $\sum  R_\mmi L_n P_\mmii = L_{nS(m)}$ and \newline
        $\sum  P_\mmi R_n L_\mmii = R_{nS(m)}$.
\end{itemize}
\end{lemma}
\begin{proof}
Parts i) and ii) are obvious. We will prove the first identities in parts iii), iv) and v). The proofs of the remaining identities are left to the reader.

iii)  Using the middle Moufang-Hopf identity,
\begin{eqnarray*}
    \sum P_\mmi P_n P_\mmii(x) & = & \sum S(\mmi)(S(\nni)(S(\mmii) x S(\mmiii))S(\nnii))S(\mmiv)\\
        &=& \sum S(\mmi)((S(\nni)S(\mmii))x(S(\nnii)S(\mmiii)))S(\mmii)\\
        &=& \sum (S(\mmi)S(\nni)S(\mmii)) x (S(\mmiii)S(\nnii)S(\mmiv))\\
        &=& \sum P_{\mmi n \mmii}(x).
\end{eqnarray*}

iv) The middle and right Moufang-Hopf identities imply
\begin{eqnarray*}
    && \hskip -.75 truein \sum  R_\mmi P_n L_\mmii (x) =  \sum (S(\nni)(\mmii x)S(\nnii))\mmi \\
     &\qquad \qquad =& \sum (S(\nni)((\mmi x \mmii)S(\mmiii))S(\nnii))\mmiv\\
     &\qquad \qquad =& \sum ((S(\nni)(\mmi x \mmii))(S(\mmiii)S(\nnii)))\mmiv\\
     &\qquad \qquad =& \sum ((S(\nni)(\mmi x \mmii))S(\mmiii))(\mmiv(S(\mmv)S(\nnii))m_{(6)})\\
    &\qquad\qquad =&\sum (S(\nni)\mmi)x(S(\nnii)\mmii)\\
    &\qquad\qquad =& P_{S(m)n}(x).
\end{eqnarray*}

v) By the right, middle and left Moufang-Hopf identities
\begin{eqnarray*}
&& \hskip -.75 truein    \sum  L_\mmi P_n R_\mmii(x) = \sum \mmi(S(\nni) ((xS(\mmii))(\mmiii 1 \mmiv)) S(\nnii))\\
     &\qquad \qquad =& \sum \mmi((S(\nni)(xS(\mmii)))((\mmiii \mmiv)S(\nnii)))\\
 &\qquad \qquad =& \sum (\mmi(S(\nni)(xS(\mmii)))\mmiii)
       (S(\mmiv)((\mmv m_{(6)}) S(\nnii)))\\
    &\qquad \qquad =& \sum ((\mmi S(\nni))x)(\mmii S(\nnii))\\
   &\qquad \qquad =& P_{nS(m)}(x).
\end{eqnarray*}
\end{proof}

Given a cocommutative Moufang-Hopf algebra $U$ we define $\Doro(U)$ as the unital associative algebra generated by abstract symbols $\{ P_m, L_m, R_m \mid m \in U\}$ subject to the relations
\begin{eqnarray*}
        & P_1 = L_1 = R_1 = 1, & \\
	& P_{\alpha m + \beta n} = \alpha P_m + \beta P_n, \quad L_{\alpha m + \beta n} = \alpha L_m + \beta L_n, \quad
	  R_{\alpha m + \beta n} = \alpha R_m + \beta R_n, & \\
        & \sum P_\mmi L_\mmii R_\mmiii = \ep(m) 1, & \\
      & \sum P_\mmi P_n P_\mmii = \sum P_{\mmi n \mmii}, \quad \sum L_\mmi L_n L_\mmii = \sum L_{\mmi n \mmii}, & \\
       & \sum R_\mmi R_n R_\mmii = \sum R_{\mmi n \mmii}, & \\
      & \sum  R_\mmi P_n L_\mmii = P_{S(m)n},\quad \sum  P_\mmi L_n R_\mmii = L_{S(m)n}, & \\
       & \sum  L_\mmi R_n P_\mmii = R_{S(m)n}, & \\
      & \sum  L_\mmi P_n R_\mmii = P_{nS(m)},\quad  \sum  R_\mmi L_n P_\mmii = L_{nS(m)} \text{ and } & \\
       & \sum  P_\mmi R_n L_\mmii = R_{nS(m),} &
\end{eqnarray*}for any $\alpha, \beta \in F$ and $m,n \in U$. The maps
\begin{displaymath}
    \Delta\colon P_m \mapsto \sum P_\mmi \otimes P_\mmii, \quad L_m \mapsto \sum L_\mmi \otimes L_\mmii, \quad R_m \mapsto \sum R_\mmi \otimes R_\mmii
\end{displaymath}
\begin{displaymath}
    \ep\colon P_m \mapsto \ep(m) 1, \quad L_m \mapsto \ep(m) 1, \quad R_m \mapsto \ep(m) 1
\end{displaymath}
\begin{displaymath}
    S\colon P_m \mapsto P_{S(m)}, \quad L_m \mapsto L_{S(m)}, \quad R_m \mapsto R_{S(m)}
\end{displaymath}
induce corresponding homomorphisms of algebras $\Delta \colon \Doro(U) \rightarrow \Doro(U) \otimes \Doro(U)$, $\ep \colon \Doro(U) \rightarrow F$ and $S \colon \Doro(U) \rightarrow \Doro(U)$ that make $\Doro(U)$ a cocommutative Hopf algebra. In fact, $\Doro(U)$ relative to the automorphisms induced by
\begin{displaymath}
    \begin{array}{ccccccc}
            &   \rho  & & & & \sigma & \\
        P_m & \mapsto & L_m & & P_m & \mapsto & P_{S(m)}\\
        L_m & \mapsto & R_m & & L_m & \mapsto & R_{S(m)}\\
        R_m & \mapsto & P_m & & R_m & \mapsto & L_{S(m)}
    \end{array}
\end{displaymath}
is a Hopf algebra with triality.

\begin{theorem}
For any cocommutative Moufang-Hopf algebra $U$, the map
\begin{eqnarray*}
    \iota \colon U &\rightarrow& \MH(\Doro(U)) \\
    m & \mapsto & P_m
\end{eqnarray*}
is an isomorphism of Moufang-Hopf algebras. Moreover, $\Doro(U)$ satisfies the following universal property: given a Hopf algebra with triality $H$ and a homomorphism $\varphi \colon U \rightarrow \MH(H)$ of Moufang-Hopf algebras, $\varphi$ extends to a homomorphism $\bar{\varphi} \colon \Doro(U) \rightarrow H$ of Hopf algebras with triality (i.e. commutes with the action of $\rho$ and $\sigma$) such that the diagram
\begin{center}
\begin{tikzpicture}
\matrix(m)[matrix of math nodes,
row sep=3em, column sep=3em,
text height=1.5ex, text depth=0.25ex]
{ & \Doro(U)\\
U & H \\ };
\path[->, font=\scriptsize] (m-2-1) edge node[above]{$\varphi$} (m-2-2);
\path[right hook->, font=\scriptsize](m-2-1) edge node[left]{$\iota$} (m-1-2);
\path[->, font=\scriptsize](m-1-2) edge node[right]{$\bar{\varphi}$} (m-2-2);
\end{tikzpicture}
\end{center}
commutes.
\end{theorem}
\begin{proof}
Let us first prove that $\iota$ is an isomorphism of Moufang-Hopf algebras. If $P_m = P_n$ then  $L_m = \rho(P_m) = \rho(P_n) = L_n$. However, Lemma \ref{lem:multiplication_algebra} implies that there exists a homomorphism from $\Doro(U)$ to the multiplication algebra of $U$ that sends $L_m$ to the left multiplication operator by $m$. Evaluating on $1$,  we obtain that $m = n$. Thus, $\iota$ is injective.  In order to check that $\iota$ is a homomorphism of algebras observe that
\begin{eqnarray*}
    \iota(m)*\iota(n) &=& P_m * P_n = \sum \rho^2(S(P_\mmi)) P_n \rho(S(P_\mmii))  \\
    &=& \sum R_{S(\mmi)} P_n L_{S(\mmii)} = P_{mn}\\
    &=& \iota(mn).
\end{eqnarray*}
By the definition of $\Delta$, $\ep$ and $S$ it is easily seen that $\iota$ is a homomorphism of Moufang-Hopf algebras. To prove that the map $\iota$ is surjective, we must check that $\MH(\Doro(U)) = \{ P_m \mid m \in U\}$. By definition,  $\MH(\Doro(U)) = \{ P(x) \mid x \in \Doro(U)\}$ but $P(xy) = \sum \sigma(\xxi)P(y)S(\xxii)$,  so we only need to prove that \begin{displaymath}
    P(P_m), P(L_m), P(R_m) \in \MH(\Doro(U))
\end{displaymath}
and that for any $m,n \in U$,
\begin{displaymath}
    \sum \sigma(P_\nni) P_m S(P_\nnii), \sum \sigma(L_\nni) P_m S(L_\nnii), \sum \sigma(R_\nni) P_m S(R_\nnii)
\end{displaymath}
also belong to $\MH(\Doro(U))$. By the definition and relations on $\Doro(U)$
\begin{eqnarray*}
    P(P_m) &=& \sum \sigma(P_\mmi)S(P_\mmii) = \sum P_{S(\mmi)}P_{S(\mmii)} = \sum P_{S(\mmi) S(\mmii)},\\
    P(L_m) &=& \sum \sigma(L_\mmi)S(L_\mmii) = \sum R_{S(\mmi)}L_{S(\mmii)} = P_m \text{ and}\\
    P(R_m) &=& \sum \sigma(R_\mmi)S(R_\mmii) = \sum L_{S(\mmi)}R_{S(\mmii)} = P_m
\end{eqnarray*}
so $P(P_m), P(L_m)$ and $P(R_m)$ belong to $\MH(\Doro(U))$. The relations on $\Doro(U)$ also imply that
\begin{eqnarray*}
 \sum \sigma(P_\nni)P_mS(P_\nnii) &=& \sum P_{S(\nni)}P_m P_{S(\nnii)} =  P_{\sum S(\nni)mS(\nnii)}, \\
 \sum \sigma(L_\nni)P_mS(L_\nnii) &=& \sum R_{S(\nni)}P_m L_{S(\nnii)} =  P_{nm} \text{ and}\\
\sum \sigma(R_\nni)P_mS(R_\nnii) &=& \sum L_{S(\nni)}P_m R_{S(\nnii)} =  P_{mn}
\end{eqnarray*}
so $\iota$ is surjective.

Let $\varphi \colon U \rightarrow \MH(H)$ be a homomorphism of Moufang-Hopf algebras where $H$ is a cocommutative Hopf algebra with triality. Given $m, n \in U$, the elements $\varphi(m), \varphi(n)$ satisfy
\begin{displaymath}
 \sum \varphi(\mmi) \rho(\varphi(\mmii)) \rho^2(\varphi(\mmiii)) = \ep(m),\\
\end{displaymath}
and by the definition of the product $*$ of $\MH(H),$
\begin{eqnarray*}
&& \sum \varphi(\mmi) \varphi(n) \varphi(\mmii) = \sum \mmi * n * \mmii,\\
&& \sum \rho^2(\varphi(\mmi))\varphi(n)\rho(\varphi(\mmii)) = \varphi(S(m))*\varphi(n) \text{ and}
\\
&& \sum \rho(\varphi(\mmi)) \varphi(n) \varphi^2(\mmii) = \varphi(n)*\varphi(S(m)).
\end{eqnarray*}
These identities and the others obtained under the action of $\rho$ show that the correspondence
\begin{displaymath}
\bar{\varphi} \colon P_m \mapsto \varphi(m), \quad L_m \mapsto \rho(\varphi(m)), \quad R_m \mapsto \rho^2(\varphi(m))
\end{displaymath}
induces a homomorphism $\bar{\varphi} \colon \Doro(U) \rightarrow H$ of Hopf algebras  with triality that makes the diagram in the statement commutative.
\end{proof}

%%%%%%%%%%%%%%%%%%%%%%%%%%%%%%%%%%%%%%
%%%%%%%%%%%%%%%%%%%%%%%%%%%%%%%%%%%%%%
%              %
%             % %
%               %
%              %
%               %
%             % %
%              %
%%%%%%%%%%%%%%%%%%%%%%%%%%%%%%%%%%%%%%
%%%%%%%%%%%%%%%%%%%%%%%%%%%%%%%%%%%%%%

\section{Lie algebras with triality}
\label{sec:Lie_algebras_with_triality}

Before any further considerations, let us present the standard example of a Lie algebra with triality. Let $\Oc = \Oc(\alpha,\beta,\gamma)$ be a generalized Cayley algebra with norm $n(\,)$ over a field of characteristic $\neq 2,3$; $\Oc_0$ the Malcev algebra of traceless elements in $\Oc$; and  $\mathfrak{o}(\Oc,n)$ the orthogonal Lie algebra of all $d\in\Endo(\Oc)$ which are skewsymmetric relative to $n(\,)$ \cite{ZSSS}.  The (local) Principle of Triality \cite{Sch95} ensures that for any $d_1\in \mathfrak{o}(\Oc,n)$ there exist unique $d_2,d_3\in \mathfrak{o}(\Oc,n)$ such that
\begin{equation}
\label{eq:ternary_derivations}
d_1(xy) = d_2(x) y + x d_3(y)
\end{equation}
for any $x,y\in \Oc$. The maps $d_1\mapsto d_2$ and $d_1\mapsto d_3$ are automorphisms of $\mathfrak{o}(\Oc,n)$. An explicit description is known.  Recall that $\mathfrak{o}(\Oc,n) = \Der(\Oc) \oplus \langle L_a \mid a\in \Oc_0\rangle \oplus \langle R_b\mid b\in \Oc_0\rangle$ and that the alternative laws $x(xy) = x^2y$, $(yx)x = y x^2$ are equivalent to the following relations
\begin{displaymath}
L_a(xy) = T_a(x)y - x L_a(y) \quad\text{and}\quad R_a(xy) = -R_a(x)y + xT_a(y)
\end{displaymath}
where $L_a$, $R_a$ denote the left and right multiplication operators by $a$, and $T_a = L_a + R_a$. Hence the automorphisms $d_1\mapsto d_2$ and $d_1\mapsto d_3$ are determined by
\begin{displaymath}
    \begin{array}{ccc}
      d & \mapsto & d \\
      L_a & \mapsto & T_a\\
      R_a & \mapsto & -R_a
    \end{array}\quad\text{and}\quad
    \begin{array}{ccc}
      d & \mapsto & d \\
      L_a & \mapsto & -L_a\\
      R_a & \mapsto & T_a
    \end{array}
\end{displaymath}
for any $d\in \Der(\Oc)$ and $a\in \Oc_0$. This proves that the maps $\rho, \sigma\colon \mathfrak{o}(\Oc,n)\to \mathfrak{o}(\Oc,n)$ given by
\begin{displaymath}
    \begin{array}{ccc}
      d & \stackrel{\rho}{\mapsto} & d \\
      L_a & \mapsto & R_a\\
      R_a & \mapsto & -T_a
    \end{array} \quad\text{and}\quad
\begin{array}{ccc}
      d & \stackrel{\sigma}{\mapsto} & d \\
      L_a & \mapsto & -R_a\\
      R_a & \mapsto & -L_a
    \end{array}
\end{displaymath}
are automorphisms, and $\mathfrak{o}(\Oc,n)$ is a Lie algebra with triality relative to $\rho$ and $\sigma$.

In this section we will prove that the universal enveloping algebra of a Lie algebra with triality is a Hopf algebra with triality. In the following, $E(\xi ;f)$ will denote the eigenspace of $f$ corresponding to the eigenvalue $\xi$.
After extending scalars if necessary,  we may assume that our base field $F$ contains a primitive cube root of unity $\omega$.

\begin{lemma}
\label{lem:P}
Let $\g$ be a Lie algebra, $\lambda \colon \Sthree_3 \to \Aut(\g)$ an action of $\Sthree_3$ as automorphisms of $\g$ and $\sigma = \lambda((12))$. Then
\begin{displaymath}
    \set{P(x)\mid x\in U(\g)} = \spann{a_n\circ(\cdots (a_2\circ a_1))\mid a_1,\dots, a_n\in E(-1;\sigma), n\in \mathbb{N}}
\end{displaymath}
where $P(x) = \sum \sigma(\xxi)S(\xxii)$ and  $a\circ x = a x + x a$.
\end{lemma}
\begin{proof}
We proceed by induction on the filtration degree of $x\in U(\g)$. In case that $x = a\in \g$, $P(a) = \sigma(a) -a \in  E(-1;\sigma)$. In general, given $a_1 \cdots a_{n+1}\in U(\g)$,  we may assume that $a_i \in E(-1;\sigma)\cup  E(1;\sigma)$,  $i = 1,\dots,n+1$. We distinguish two cases:
\begin{itemize}
\item[i)] \emph{At least one $a_i$ belongs to $E(1;\sigma)$}: in this case, up to terms of lower degree we may assume that $a_{n+1}\in E(1;\sigma)$. With $x= a_1\cdots a_{n}$, $a = a_{n+1}$ we obtain
    \begin{displaymath}
    P(xa) = \sum \sigma(\xxi) \sigma(a) S(\xxii) - \sum \sigma(\xxi) a S(\xxii)  = 0
    \end{displaymath}
\item[ii)]  $a_1,\dots, a_{n+1}\in E(-1;\sigma)$: In this case, with $x = a_2\cdots a_{n+1}$ and $a=a_{1}$ we have
\begin{displaymath}
P(ax) = \sum \sigma(a)\sigma(\xxi)S(\xxii) - \sum \sigma(\xxi)  S(\xxii) a = -a\circ P(x)
\end{displaymath}
and the result follows by induction.
\end{itemize}
\end{proof}

\begin{lemma}
\label{lem:eigen_one}
Given a Lie algebra $\g$ over a field of characteristic $\neq 2,3$ and two automorphisms $\sigma,\rho$ with $\sigma^2 = \rho^3=\Id_\g$ and $\sigma \rho = \rho^2 \sigma$, then $\g$ is a Lie algebra with triality relative to $\rho$ and $\sigma$ if and only if $E(1;\rho)\subseteq E(1;\sigma)$.
\end{lemma}
\begin{proof}
Given $a\in E(1;\rho)$, (\ref{eq:triality_Lie}) implies that $3\sigma(a) - 3a = 0$ so $\sigma(a) = a$.  Conversely, for elements $a\in E(1;\rho)$ condition (\ref{eq:triality_Lie}) follows from our hypothesis, while for elements $a\in E(\omega;\rho)$, we have that $\sum_{\tau\in \Sthree_3} \sig(\tau)\tau(a) = (1+\omega + \omega^2) (a - \sigma(a)) = 0$. Since $\g = E(1;\rho)\oplus E(\omega;\rho)\oplus E(\omega^2;\rho)$ we are done.
\end{proof}

The universal enveloping algebra $U(\g)$ of a Lie algebra $\g$ acts on $\g$ by
\begin{displaymath}
    x\cdot a = \sum x_{(1)}a S(x_{(2)})
\end{displaymath}
where $S$ denotes the antipode. With this notation we have:

\begin{lemma}
\label{lem:basic_equation}
Let $\g$ be a Lie algebra with triality relative to $\rho$ and $\sigma$ over a field of characteristic $\neq 2,3$. Then $U(\g)$ satisfies
\begin{eqnarray}
    \label{eq:action}
    && \ep(x) a - P(x)\cdot \sigma(a) + P(x)\cdot \rho(a) \\
    \nonumber && \quad -\rho^2\sigma(P(x))\cdot \rho\sigma(a)+ \rho^2\sigma(P(x)) \cdot \rho^2(a) - \ep(x)\rho^2(\sigma(a))=0.
\end{eqnarray}
\end{lemma}
\begin{proof}
First observe that $\sigma \rho = \rho^2\sigma$ implies that
\begin{displaymath}
    \sigma(E(\omega;\rho)) = E(\omega^2;\rho).
\end{displaymath}
In particular, by Lemma \ref{lem:eigen_one} any element $a\in E(-1;\sigma)$ can be written as $a = a' - a''$ with $a'\in E(\omega;\rho)$ and $a'' = \sigma(a')$. Also observe that by Lemma \ref{lem:P} we may assume that $P(x) = a_n\circ (\cdots (a_2\circ a_1))$ with $a_1,\dots, a_n \in E(-1;\sigma)$. This implies that $\sigma(P(x)) = (-1)^n P(x)$.

By Lemma \ref{lem:eigen_one} the result is obvious if $a\in E(1;\rho)$, so we may assume that $a\in E(\omega;\rho) \cup E(\omega^2;\rho)$. In fact,  we only have to prove (\ref{eq:action}) for $a=a'\in E(\omega;\rho)$ because the case $a\in E(\omega^2;\rho)$ is a consequence of the former by applying $\sigma$ to $(\ref{eq:action})$.

Let us denote $p = a_n\circ (\cdots (a_2\circ a_1))$. If $n = 0$ then $p = 1$ and (\ref{eq:action}) is (\ref{eq:triality_Lie}), so we may assume that $n\geq 1$. Equation (\ref{eq:action}) can be written as
\begin{equation}
    \label{eq:action_simplified}
    - p\cdot a'' + \omega p\cdot a' -(-1)^n\omega^2\rho^2(p)\cdot a''+ (-1)^n \omega^2\rho^2(p) \cdot a' =0
\end{equation}
with $a'' = \sigma(a')$.

To simplify computations we will write $p_0, p', p''$ for the projections of $p$ on $E(1;\rho), E(\omega;\rho), E(\omega^2;\rho)$ where $\rho$ also denotes the extension of $\rho$ to an automorphism of $U(\g)$. Observe that $ \sigma(p)=(-1)^n p$ implies that $\sigma(p_0) = (-1)^n p_0$, $\sigma(p') = (-1)^n p''$ and $\sigma(p'') = (-1)^n p'$.

Plugging $p = p_0 + p' + p''$ into (\ref{eq:action_simplified}) and taking projections onto $E(1;\rho), E(\omega;\rho)$ and $E(\omega^2;\rho)$ we get that (\ref{eq:action_simplified}) is equivalent to
\begin{itemize}
\item[(i)] $-\omega p' \cdot a'' + \omega^2 p'' \cdot a' = -(-1)^n\left( -\omega^2 p'\cdot a'' + \omega p''\cdot a' \right)$,
\item[(ii)] $p_0 \cdot a' - \omega^2 p''\cdot a'' = -(-1)^n \left( \omega p_0\cdot a' - \omega^2 p''\cdot a''\right)$ and
\item[(iii)] $\omega p'\cdot a' - p_0 \cdot a'' = -(-1)^n \left( \omega p'\cdot a' - \omega^2 p_0\cdot a''\right)$.
\end{itemize}
Therefore, we need to prove (i), (ii) and (iii).

In order to prove (i) we notice that $-\omega p' \cdot a'' + \omega^2 p'' \cdot a' \in E(1;\rho)\cap \g$. By Lemma \ref{lem:eigen_one} we get $-\omega p' \cdot a'' + \omega^2 p'' \cdot a' = \sigma(-\omega p' \cdot a'' + \omega^2 p'' \cdot a') = -(-1)^n \omega p''\cdot a' + (-1)^n \omega^2 p'\cdot a''$, as desired.

Equality (iii) follows from (ii) by applying $\sigma$ to it.

We split the proof of (ii) into two cases depending on the parity of $n$. If $n$ is odd then (ii) is equivalent to
\begin{equation}
    p_0 \cdot a' = 0, \tag{odd}
\end{equation}
while in case that $n$ is even,  (ii) is equivalent to
\begin{equation}
    p_0 \cdot a' = - 2 p''\cdot a'' \tag{even}
\end{equation}
We will use induction on $n$. Let us assume that (\ref{eq:action_simplified}) holds for $p$ and consider $q = b\circ p$ with $b = b'-b''\in E(-1;\sigma)$, $b'\in E(\omega;\rho)$ and $b'' = \sigma(b')$. The projections of $q$ on $E(1;\rho), E(\omega;\rho)$ and $E(\omega^2;\rho)$ are
\begin{eqnarray*}
    q_0 &=& b'p'' + p'' b' - b''p' -p'b'', \\
    q' &=& b' p_0 + p_0 b' - b'' p'' - p''b'' \text{ and}\\
    q'' &=& -p_0 b'' - b'' p_0 + b' p' + p' b'
\end{eqnarray*}
respectively.

Let us assume that $n$ is odd. In this case we should prove that $q_0 \cdot a' = - 2 q'' \cdot a''$. First observe that the hypothesis of induction implies that $p_0\cdot a' = 0$. Applying $\sigma$ to both sides of this equality we also obtain that $p_0 \cdot a'' = 0$. Equality (i) implies that $p'\cdot a'' = - p''\cdot a'$. The commutator $[b'',a']$ belongs to $E(1;\rho)$ so, by Lemma \ref{lem:eigen_one}, $[b'',a'] = [b',a'']$. With this information we compute $q_0 \cdot a' + 2 q'' \cdot a''$ as follows
\begin{eqnarray*}
    q_0 \cdot a' + 2 q'' \cdot a'' &=& (b' p'' + p''b' - b'' p'- p'b'')\cdot a' + 2(b' p' + p'b') \cdot a'' \\
    &=& (p'' b' - b'' p')\cdot a' + (b' p' + p'b')\cdot a'' \\
    &=& (p'' b' - b'' p')\cdot a' +(-b' p'' + p'b'')\cdot a'\\
    &=& -([b',p'']+[b'',p'])\cdot a'.
\end{eqnarray*}
Since $[b'+b'',a_i]\in E(-1;\sigma)$ and $[b'+b'',p] = \sum_{i=1}^n a_n\circ(\cdots ( [b'+b'',a_i]\circ(\cdots (a_2\circ a_1))))$ by (even) the action of the projection of $[b'+b'',p]$ on $E(1;\rho)$ kills $a'$. This projection is clearly $[b',p'']+[b'',p']$ so $([b',p'']+[b'',p'])\cdot a'=0$. This proves that $q_0 \cdot a' = - 2 q'' \cdot a''$.

Let us assume now that $n$ is even. In this case we should prove that $q_0 \cdot a' = 0$. The hypothesis of induction and (i), (ii) and (iii) for $p$ imply that $p'\cdot a'' = p''\cdot a'$, $p_0\cdot a' = -2 p''\cdot a''$ and $p_0 \cdot a'' = - 2 p'\cdot a'$. The hypothesis of induction applied to $[b'+b'',p] = ([b',p''] + [b'',p']) + ([b',p_0]+[b'',p'']) + ([b'',p_0] + [b',p'])$ also implies that $([b',p''] + [b'',p'])\cdot a' = - 2([b',p'] + [b'',p_0])\cdot a''$. Hence
\begin{eqnarray*}
    0&=& ([b',p''] + [b'',p'])\cdot a' + 2([b',p'] + [b'',p_0])\cdot a''\\
    &=& 3 b' p'' \cdot a' - p'' b'\cdot a' + b'' p'\cdot a' - p' b''\cdot a' - 2p'b'\cdot a'' - 4 b'' p'\cdot a' + 4 p'' b'\cdot a' \\
    &=& 3 b' p''\cdot a' + 3 p'' b' \cdot a' - 3 b'' p'\cdot a' - 3 p' b''\cdot a'\\
    &=& 3q_0\cdot a'.
\end{eqnarray*}
Since the characteristic of the base field is $\neq 2,3$ then $q_0\cdot a' = 0$. This concludes the proof.
\end{proof}

\begin{theorem}\label {thm:ugtriality}
Let $\g$ be a Lie algebra with triality relative to $\rho$ and $\sigma$  over a field of characteristic $\neq 2,3$. Then $U(\g)$ is a Hopf algebra with triality relative to $\rho$ and $\sigma$.
\end{theorem}
\begin{proof}
We will prove (\ref{eq:triality_Hopf}) by induction on the filtration degree of $x\in U(\g)$. If this degree is $0$ then (\ref{eq:triality_Hopf}) holds trivially.  So, let us assume that (\ref{eq:triality_Hopf}) holds for $x\in U(\g)$ of filtration degree $\leq n$ and let us prove that it also is valid for $ax$ with $a\in \g$. Recall that $P(ax) = \sigma(a) P(x)  - P(x)a$.

Since by induction we have  that $\sum P(x_{(1)})\rho(P(x_{(2)}))\rho^2(P(x_{(3)})) = \ep(x)1$ then  $\sum \rho(P(\xxi)) \rho^2(P(\xxii)) = S(P(x))$ and $\sum P(\xxi)\rho(P(\xxii)) = S(\rho^2(x))$. Hence

\begin{eqnarray*}
   && \hskip -.47 truein  \sum P((ax)_{(1)})\rho(P((ax)_{(2)}))\rho^2(P((ax)_{(3)})) \\
   && \quad = \sum P(a \xxi) \rho(P(\xxii)) \rho^2(P(\xxiii)) + P(\xxi ) \rho(P(a \xxii)) \rho^2(P(\xxiii)) \\
   && \quad\quad + \sum P(\xxi ) \rho(P(\xxii)) \rho^2(P(a \xxiii)) \\
   && \quad = \sum \left(\sigma(a) P(\xxi)  -P(\xxi)a \right) S(P(\xxii)) \\
   && \quad\quad + P(\xxi) \rho(\sigma(a) P(\xxii) - P(\xxii) a)\rho^2(P(\xxiii)) \\
   && \quad\quad + \sum S(\rho^2(P(\xxi))) \rho^2(\sigma(a) P(\xxii) - P(\xxii)a)\\
\end{eqnarray*}
\begin{eqnarray*}
   && \quad = \ep(x) \sigma(a) -  \sum P(\xxi) a S(P(\xxii)) \\
   && \quad\quad  + \sum P(\xxi) \rho\sigma(a) S(P(\xxii)) - \sum S(\rho^2(P(\xxi))) \rho(a) \rho^2(P(\xxii))\\
   && \quad\quad + \sum S(\rho^2(P(\xxi))) \rho^2\sigma(a)\rho^2(P(\xxii)) - \ep(x) \rho^2(a)\\ %
   && \quad = \ep(x)\sigma(a) - P(x)\cdot a + P(x)\cdot \rho\sigma(a) - \rho^2\sigma(P(x)) \cdot \rho(a) \\
   && \quad\quad + \rho^2\sigma(P(x))\cdot \rho^2\sigma(a) - \ep(x) \rho^2(a).
\end{eqnarray*}
Lemma \ref{lem:basic_equation} applied to $\sigma(a)$ then implies
\begin{displaymath}
    \sum P((xa)_{(1)})\rho(P((xa)_{(2)}))\rho^2(P((xa)_{(3)}))  = 0 = \ep(xa)1.
\end{displaymath}
\end{proof}

%%%%%%%%%%%%%%%%%%%%%%%%%%%%%%%%%%%%%%
%%%%%%%%%%%%%%%%%%%%%%%%%%%%%%%%%%%%%%
%              %
%             %%
%            % %
%           %  %
%          %%%%%%%
%              %
%              %
%%%%%%%%%%%%%%%%%%%%%%%%%%%%%%%%%%%%%%
%%%%%%%%%%%%%%%%%%%%%%%%%%%%%%%%%%%%%%

\section{The universal enveloping algebra of a Malcev algebra}
\label{sec:universal_enveloping_algebra}

In \cite{PeSh04},  the construction of the universal enveloping algebra of a Malcev algebra $\m$ over a field $F$ of characteristic $\neq 2, 3$ begins by attaching a Lie algebra $\Lie(\m)$ to the Malcev algebra $\m$. This Lie algebra is the Lie algebra generated by the symbols $\{ \lambda_a, \rho_a \mid a \in \m\}$ subject to the relations
\begin{equation}
 \label{eq:PeSh_relations}
 \begin{matrix}
    \lambda_{\alpha a + \beta b} = \alpha \lambda_a + \beta \lambda_b \hfill & \rho_{\alpha a + \beta b} = \alpha \rho_a + \beta \rho_b \hfill\\
    [\lambda_a, \lambda_b] = \lambda_{[a,b]} - 2 [\lambda_a,\rho_b] \hfill & [\rho_a,\rho_b] = - \rho_{[a,b]} - 2 [\lambda_a, \rho_b] \hfill\\
    [\lambda_a, \rho_b] = [\rho_a, \lambda_b] \hfill&
 \end{matrix}
\end{equation}
for any $a, b \in \m$, $\alpha, \beta \in F$. With the notation
\begin{displaymath}
 \ad_a = \lambda_a - \rho_a,\quad  T_a = \lambda_a + \rho_a \quad \text{and} \quad D_{a,b} = \ad_{[a,b]} - 3 [\lambda_a, \rho_b]
\end{displaymath}
 it was proved that $\Lie(\m) = \Lie_+ \oplus \Lie_-$ with $\Lie_+ = \spann{\ad_a, D_{a,b} \mid a,b \in \m}$ and $\Lie_- = \spann{T_a \mid a \in \m}$;  the mapping $T_a \mapsto a$ being a linear isomorphism from $\Lie_-$ onto $\m$ \cite{PeSh04}*{Proposition 3.2}. The underlying vector space of $U(\m)$ was the symmetric algebra
 ${\mathcal S}(\Lie_-)$ on $\Lie_-$, where a structure of an $\Lie(\m)$-module was given. It also was observed that $\Lie(\m)$ admits two automorphisms $\zeta, \eta$ determined by
\[
\begin{array}{ll}
 \zeta(\lambda_a) = T_a & \eta(\lambda_a) = -\lambda_a \\
 \zeta(\rho_a) = - \rho_a & \eta(\rho_a) = T_a
\end{array}
\]
With these automorphisms,  the structure of the $\Lie(\m)$-module of ${\mathcal S}(\m)$ ($\Lie_-$ is identified with $\m$) is twisted to get two new modules ${\mathcal S}(\m)_\zeta$, ${\mathcal S}(\m)_\eta$. The product of $U(\m)$ is obtained as  a homomorphism of $\Lie(\m)$-modules ${\mathcal S}(\m)_\zeta \otimes {\mathcal S}(\m)_\eta \rightarrow {\mathcal S}(\m)$ satisfying certain conditions \cite{PeSh04}*{Proposition 3.4}.

\begin{proposition}
 Let $\m$ be a Malcev algebra over a field of characteristic $\neq 2,3$. Then $\Lie(\m)$ is a Lie algebra with triality relative to $\rho = \eta \zeta$ and $\sigma = \zeta \eta \zeta$.
\end{proposition}
\begin{proof}
The automorphisms $\rho, \sigma$ are determined by their action on the generators
 \[
\begin{array}{ll}
 \sigma(\lambda_a) = - \rho_a & \rho(\lambda_a) = \rho_a \\
 \sigma(\rho_a) = - \lambda_a & \rho(\rho_a) = - T_a
\end{array}
\]
Thus they clearly satisfy $\sigma^2 = \Id_{\Lie(\m)} = \rho^3$ and $\sigma \rho = \rho^2 \sigma$.  Because of (\ref{eq:PeSh_relations}), the elements $\ad_a, D_{a,b}$ are fixed by $\sigma$ so it suffices to check (\ref{eq:triality_Lie}) for elements $T_a$, but this is obvious.
\end{proof}

By Theorem \ref{thm:ugtriality}, the universal enveloping algebra $U(\Lie(\m))$ of $\Lie(\m)$ is then a Hopf algebra with triality,  so we can consider the Moufang-Hopf algebra $\MH(U(\Lie(\m)))$.

\begin{lemma}
\label{lem:aux1}
 For any $a \in \m$ and $u \in \MH(U(\Lie(\m)))$ the elements  $T_a * u + u * T_a \in \MH(U(\Lie(\m)))$ and $T_a u + u T_a \in U(\Lie(\m))$ coincide.
\end{lemma}
\begin{proof}
We use both formulas for the product on $\MH(\Lie(\m))$ in Theorem \ref{thm:product_MH}. On the one hand $T_a * u = \rho^2(S(T_a))u + u\rho(S(T_a)) = \rho_a u + u \lambda_a$. On the other hand, $u*T_a = \rho(S(T_a)) u + u \rho^2(S(T_a)) = \lambda_a u + u \rho_a$. Thus  $T_a* u + u*T_a = T_a u + u T_a$.
\end{proof}

\begin{theorem}
\label{thm:main}
Let $\m$ be a Malcev algebra over a field of characteristic $\neq 2,3$. Then $U(\m)$ is isomorphic to $\MH(U(\Lie(\m)))$.
\end{theorem}
\begin{proof}
 Lemma \ref{lem:P} and Lemma \ref{lem:aux1} imply that $\MH(U(\Lie(\m)))$ is spanned by elements of the form $T_{a_n} \bullet (\cdots (T_{a_2}\bullet T_{a_1}))$ with $a_1,\dots, a_m \in \m$, $n \in \mathbb{N}$ and $T_a \bullet u = T_a * u + u * T_a$. The elements $T_a$ are primitive inside the Moufang-Hopf algebra $\MH(U(\Lie(\m)))$, so they belong to the generalized alternative nucleus. The commutator of two of them in $\MH(U(\Lie(\m)))$ is easily computed:
\begin{eqnarray*}
 T_a*T_b &=& \rho^2(S(T_a))T_b + T_b \rho(S(T_a)) = \rho_a T_b + T_b \lambda_a,\\
T_a*T_b - T_b*T_a &=& \rho_a\rho_b + \lambda_b \lambda_a - \rho_b\rho_a - \lambda_a\lambda_b = [\rho_a, \rho_b] - [\lambda_a,\lambda_b] = - T_{[a,b]}
\end{eqnarray*}
where the last equality follows from relations (\ref{eq:PeSh_relations}). By the universal property of $U(\m)$ in \cite{PeSh04},  we can conclude that the map $a \mapsto -T_a$ extends to a homomorphism of unital algebras $\varphi \colon U(\m) \rightarrow \MH(U(\Lie(\m)))$. In fact, since $\MH(U(\Lie(\m))) = \spann{T_{a_n} \bullet (\cdots (T_{a_2} \bullet T_{a_1})) \mid a_1, \dots, a_n \in \m, n \in \mathbb{N}}$,  then $\varphi$ is surjective. To prove that $\varphi$ is an isomorphism we appeal to the Poincar\'e-Birkhoff-Witt Theorem for $U(\m)$ and $U(\Lie(\m))$.
Given a totally ordered basis $\{ a_i \}_{i\in \Lambda}$ of $\m$, $U(\m)$ admits a basis $\{ a_{i_n} \bullet (\cdots (a_{i_2} \bullet a_{i_1}))    \mid i_1 \leq \cdots \leq i_n, n\in \mathbb{N} \}$ with $a \bullet x = a x + x a$ in $U(\m)$ \cite{PeSh04}. This basis is sent to $\{ T_{a_{i_n}} \circ (\cdots (T_{a_{i_2}} \circ T_{a_{i_1}}))    \mid i_1 \leq \cdots \leq i_n, n\in \mathbb{N} \}$, a linearly independent set in $U(\Lie(\m))$. Since $\varphi$ sends a basis to a linearly independent set, then $\varphi$ is also injective.
\end{proof}

\begin{theorem}
 Let $\m$ be a Malcev algebra over a field of characteristic $\neq 2,3$.
Then $\Doro(U(\m))$ is isomorphic to $U(\Lie(\m))$.
\end{theorem}
\begin{proof}
 The isomorphism $\varphi\colon U(\m) \rightarrow \MH(U(\Lie(\m)))$ in the proof of Theorem \ref{thm:main} sends $a \in \m$ to $-T_a$. By the universal property of $\Doro(U(\m))$ we obtain a homomorphism $\bar{\varphi} \colon \Doro(U(\m)) \rightarrow U(\Lie(\m))$ of Hopf algebras with triality that sends $P_a$ to $-T_a$. The relations that define $\Doro(U(\m))$ imply that most of the generators $\{ P_m, L_m, R_m \mid m \in U(\m) \}$ are superfluous. In fact, since $U(\m)$ is generated by $\m$, then $\Doro(U(\m))$ is generated by $\{ L_a, R_a \mid a \in \m \}$. The images of these generators under $\bar{\varphi}$ are $\bar{\varphi}(L_a) = \bar{\varphi}(\rho(P_a)) =  \rho(\bar{\varphi}(P_a)) = - \rho(T_a) = \lambda_a$ and $\bar{\varphi}(R_a) = \rho_a$.

The identities that define $\Doro(U(\m))$ also imply that
\begin{eqnarray*}
 -L_{ab} &=& P_aL_b + L_bR_a = -L_aL_b - [R_a,L_b]\\
 -L_{ab} &=& R_bL_a + L_aP_b = -L_aL_b - [L_a,R_b]
\end{eqnarray*}
so $[L_a,R_b] = [R_a,L_b]$ and $[L_a,L_b] = L_{[a,b]} - 2 [L_a,R_b]$. In a similar way $[R_a,R_b] = - R_{[a,b]} - 2 [L_a,R_b]$. This proves that the generators $\{ L_a, R_a \mid a \in \m\}$ satisfy similar relations to those in (\ref{eq:PeSh_relations}). Therefore, there exists a homomorphism of Hopf algebras $U(\Lie(\m)) \rightarrow \Doro(U(\m))$ that sends $\lambda_a$ to $L_a$ and $\rho_a$ to $R_a$. This homomorphism is clearly the inverse of $\bar{\varphi}$, so $\bar{\varphi}$ is an isomorphism of Hopf algebras.
\end{proof}

%%%%%%%%%%%%%%%%%%%%%%%%%%%%%%%%%%%%%%
%%%%%%%%%%%%%%%%%%%%%%%%%%%%%%%%%%%%%%
%  %%%%%%%   %      %   %%%
%  %         % %    %   %  %
%  %         %  %   %   %   %
%  %%%%%%%   %   %  %   %   %
%  %         %    % %   %   %
%  %         %     %%   %  %
%  %%%%%%%   %      %   %%%
%%%%%%%%%%%%%%%%%%%%%%%%%%%%%%%%%%%%%%
%%%%%%%%%%%%%%%%%%%%%%%%%%%%%%%%%%%%%%

\section{Moufang loops from coalgebras morphisms}
\label{sec:other_constructions}

Let $(\C,\Delta, \ep)$ be a cocommutative coalgebra and $U$ a Moufang-Hopf algebra  that will remain fixed throughout this section. It is known \cite{Pe07} that the set of coalgebra morphisms $\morco(\C,U)$ from $\C$ to $U$ is a Moufang loop with the convolution product $f*g$ given by
\begin{displaymath}
c(f*g) = \sum(c_{(1)}f)(c_{(2)}g),
\end{displaymath}
the unit element being $c \mapsto \ep(c) 1$. (Recall, here we have reverted to writing operators on the right as in Section 2.) This is a nonassociative analog of the fact that for any Hopf algebra $H$,  $\morco(\C,H)$ is a group under the convolution.

In the vector space $\Hom(\C,\Endo(U))$ we may also define a convolution product
\begin{displaymath}
    (A*B)_c = \sum A_{c_{(1)}}B_{c_{(2)}}
\end{displaymath}
for which $\Hom(\C,\Endo(U))$ is an associative algebra with identity $c \mapsto \ep(c) \Id_U$. The notation for the image $A_c$ of $c$ under $A$ is consistent  with the notation $L_x$, $R_x$ for the multiplication operators on $U$ that we can interpret as elements $L,R \in \Hom(U, \Endo(U))$. Consider the elements $A \in \Hom(\C,\Endo(U))$ such that
\begin{enumerate}
    \item $A$ is invertible,
    \item $\ep(y A_x) = \ep(y) \ep(x)$ and
    \item $\Delta(y A_x) = \sum y_{(1)} A_{x_{(1)}} \otimes y_{(2)} A_{x_{(2)}}$
\end{enumerate}
and call $G = G(\C,U)$ the set of all of them. Clearly $G$ is a group. For instance, if $\C = U$ then $L \colon x \mapsto L_x$, $R\colon x \mapsto R_x$ belong to $G(U,U)$. In fact,
\begin{eqnarray*}
    & \sum yL_\xxi L_{\xxii S} = \ep(x) y =   \sum y L_{\xxi S} L_\xxii & \text{and}\\
    & \sum yR_\xxi R_{\xxii S} = \ep(x)y =  \sum y R_{\xxi S} R_\xxii &
\end{eqnarray*}
show that $L$ and $R$ are invertible in $\Hom(U,\Endo(U))$ with inverses
\begin{displaymath}
    L^{-1}\colon x \mapsto L_{xS} \quad\text{and}\quad R^{-1}\colon x \mapsto R_{xS}.
\end{displaymath}
The map $U = L*R \colon x \mapsto \sum L_\xxi R_\xxii$ also belongs to $G(U,U)$.

Let us define now
\begin{displaymath}
    \Atp_\C(U) = \{ (A,B,C) \in G^3 \mid (xy) A_{c} = \sum (xB_{c_{(1)}}) (yC_{c_{(2)}}) \quad \forall_{x,y \in U,\, c\in \C}\}.
\end{displaymath}
To justify our notation observe that  when $U = FQ$ is the loop algebra of a Moufang loop $Q$ and $\C = Fe$ is the one-dimensional coalgebra with $\Delta(e) = e \otimes e$ and $\ep(e) = 1$, then $G$ can be identified with $\Bij(Q)$ and $\Atp_\C(U)$ with $\Atp(Q)$.

The goal of this section is to prove that, in analogy with results in Section \ref{sec:groups_with_triality}, $\Atp_\C(U)$ is a group with triality for which $\M(\Atp_\C(U)) = \morco(\C,U)$.

\begin{proposition}
$\Atp_\C(U)$ is a group under  the componentwise product.
\end{proposition}
\begin{proof}
It is easy to check that $\Atp_\C(U)$ is closed under products, so we only have to prove that $(A^{-1}, B^{-1}, C^{-1}) \in \Atp_\C(U)$ whenever $(A, B, C) \in \Atp_\C(U)$. Let $(A, B, C) \in \Atp_\C(U)$. On the one hand,
\begin{displaymath}
    \sum (((xB^{-1}_\ci)(yC^{-1}_\cii))A_\ciii)A^{-1}_\civ = \sum (xB^{-1}_\ci)(yC^{-1}_\cii);
\end{displaymath}
while on the other
\begin{eqnarray*}
    \sum (((xB^{-1}_\ci)(yC^{-1}_\cii))A_\ciii)A^{-1}_\civ &=&
    \sum (((xB^{-1}_\ci)B_\cii)((yC^{-1}_\ciii)C_\civ))A^{-1}_\cv\\
    &=& (xy)A^{-1}_c.
\end{eqnarray*}
\end{proof}

Given $A \in G$ consider
\begin{displaymath}
    L_A\colon c \mapsto L_{1A_{c}},\quad R_A\colon c \mapsto R_{1A_{c}},\quad U_A\colon c \mapsto U_{1A_{c}}
\end{displaymath}
where $yU_x = \sum x_{(1)}yx_{(2)}$ for any $x, y \in U$. These maps are invertible with inverses
\begin{displaymath}
    (L_A)^{-1} = (L^{-1})_A, \quad (R_A)^{-1} = (R^{-1})_A, \text{ and } (U_A)^{-1} = (U^{-1})_A.
\end{displaymath}
In fact, $L_A, R_A, U_A \in G(\C,U)$.

\begin{lemma}
For any $A \in G(\C,U)$  we have that
\begin{displaymath}
    (L_A,U_A,L^{-1}_A), (R_A,R^{-1}_A,U_A), (U_A,L_A,R_A) \in \Atp_\C(U).
\end{displaymath}
\end{lemma}
\begin{proof}
This statement is a consequence of the Moufang-Hopf identities. For instance, the left Moufang-Hopf identity implies that
\begin{displaymath}
   (xy)(L_A)_c \hskip -1pt = \sum (xy)L_{1A_c} \hskip -1pt = \sum (xU_{1A_{\ci}})(yL_{1A_{\cii}S}) \hskip -1pt = \sum (x (U_A)_\ci)(y(L^{-1}_A)_\cii).
\end{displaymath}
hence $(L_A,U_A,L^{-1}_A) \in \Atp_\C(U)$.
\end{proof}

We leave the proof of the following lemma as an exercise.
\begin{lemma}
Let $B \in G(\C,U)$ be such that $1 B_c = \ep(c) 1$ for any $c \in \C$. Then $L_B = R_B = U_B = 1_{G(\C,U)}$.
\end{lemma}

\begin{lemma}
\label{lem:middle_Hopf}
If  $(A, B, C) \in \Atp_\C(U)$, then $A = B * R_C$ and $A = C*L_B$. In particular, if $1B_c = \ep(c) 1$ for all $c\in \C$ then $A = C$.
\end{lemma}
\begin{proof}
Since $(xy)A_c = \sum (xB_\ci)(yC_\cii)$, then evaluating at $y = 1$ we obtain $xA_c = \sum (xB_\ci)(1C_\cii) = (xB_\ci)R_{1C_{\cii}} = x(B*R_C)_c$. The other equality in the statement follows from evaluating at $x = 1$.
\end{proof}

\begin{lemma}
\label{lem:decomposition}
Any $(A,B,C) \in \Atp_\C(U)$ can be written as
\begin{displaymath}
(A, B, C) = (D',D,D')(R^{-1}_B, R_B, U^{-1}_B)
\end{displaymath}
for some $(D',D,D')\in \Atp_\C(U)$.
\end{lemma}
\begin{proof}
The only possibility for $(D',D,D')$ is $(A*R_B,B*R^{-1}_B,C*U_B)$ once we have demonstrated that $A*R_B = C*U_B$; but this relation is a consequence of  $1(B*R^{-1}_B)_c =\sum (1B_\ci)((1B_\cii)S) = \ep(c) 1$ and Lemma \ref{lem:middle_Hopf}.
\end{proof}

Given $A \in G(\C,U)$ consider
\begin{displaymath}
A^S \colon c \mapsto SA_c S.
\end{displaymath}
This map is invertible with inverse $(A^{-1})^S$:
\begin{displaymath}
    ((A^{-1})^S*A^S)_c = \sum SA^{-1}_\ci S S A_\cii S = \ep(c) \Id_U = (A^S*(A^{-1})^S)_c.
\end{displaymath}
In fact, it can be easily seen that $A^S \in G(\C,U)$ and that $A \mapsto A^S$ is an involutive automorphism of $G(\C,U)$. In the following we will write $A^{-S}$ instead of $(A^{-1})^S$.

\begin{proposition}
Let $U$ be a cocommutative Moufang-Hopf algebra and $\C$ a cocommutative coalgebra. Then $\Atp_\C(U)$ is a group with triality relative to the automorphisms $\rho, \sigma$ given by
\begin{displaymath}
    (A,B,C)^\rho = (B^S,C,A^S) \quad\text{and}\quad (A,B,C)^\sigma = (C, B^S, A)
\end{displaymath}
for any $(A, B, C) \in \Atp_\C(U)$.
\end{proposition}
\begin{proof}
 We can compute $\sum ((\xxi S)B_\ci S)(((\xxii S) (\xxiii y))A_\cii)$ either as
\begin{displaymath}
    \sum (x SB_\ci S)(yA_\cii)
\end{displaymath}
or as
\begin{displaymath}
    \sum (\xxi SB_\ci S)((\xxii SB_\cii) ((\xxiii y)C_\ciii))
    = (xy)C_c.
\end{displaymath}
This proves that $(C, B^S, A) \in \Atp_\C(U)$. A similar computation for
\begin{displaymath}
    \sum ((x y_{(1)})(y_{(2)}S))A_\ci (y_{(3)}SC_\cii S)
\end{displaymath}
gives $(B,A,C^S) \in \Atp_\C(U)$. Starting with $(C,B^S,A)$ instead of $(A,B,C)$ we get $(B^S,C,A^S) \in \Atp_\C(U)$. The maps $\rho \colon (A,B,C) \mapsto (B^S,C,A^S)$ and $\sigma \colon (A,B,C) \mapsto (C,B^S,A)$ so defined verify the relations  $\sigma^2 = \Id_{\Atp_\C(U)} = \rho^3$, $\sigma \rho = \rho^2 \sigma$.

At this point we should  observe that if $(A,B,A) \in \Atp_\C(U)$,  then $(A,B^S,A) = (A,B,A)^\sigma \in \Atp_\C(U)$ so $(1_{G(\C,U)},B^{-1}*B^S,1_{G(\C,U)}) \in \Atp_\C(U)$, i.e., $B^{-1}*B^S = 1_{G(\C,U)}$. Therefore $B^S = B$.

The proof of relation (\ref{eq:triality_Group}) for $\Atp_\C(U)$ is similar to that of Theorem \ref{thm:GrZa}. The equalities to be checked here are
\begin{eqnarray*}
    A^{-1}*C*B^{-S}*B*C^{-S}*A^S &=& 1_{G(\C,U)},\\
    B^{-1}*B^S*C^{-1}*A*A^{-S}*C^S &=& 1_{G(\C,U)} \text{ and}\\
    C^{-1}*A*A^{-S}*C^S*B^{-1}*B^S &=& 1_{G(\C,U)}.
\end{eqnarray*}
We will only prove the first one. By Lemma \ref{lem:middle_Hopf} we can write $A = C*L_B$ so $A^{-1}*C*B^{-S}*B*C^{-S}*A^S = L^{-1}_B*B^{-S}*B*L^S_B$. Now by Lemma \ref{lem:decomposition}   we decompose $(A,B,C)$ as $(D',D,D')(R^{-1}_B,R_B,U^{-1}_B)$ to obtain that $B = D*R_B$ with $D^S = D$. Thus, $A^{-1}*C*B^{-S}*B*C^{-S}*A^S = L^{-1}_B*B^{-S}*B*L^S_B = L^{-1}_B*R^{-S}_B*R_B*L^S_B = 1_{G(\C,U)}$ since $L^S_B = R^{-1}_B$.
\end{proof}

The loop $\M(\Atp(\C,U))$ consists of the elements
\begin{displaymath}
    (A,B,C)^{-1}(A,B,C)^\sigma = (A^{-1}*C,B^{-1}*B^S,C^{-1}*A) = (L^{-1}_B,U^{-1}_B,L_B).
\end{displaymath}
In fact, for any $B \in G(\C,U)$ we have that $(U_B,L_B,R_B) \in \Atp(\C,U)$ and $L_{L_B} = L_B$,  so we can identify $\M(\Atp(\C,U))$ with $\{ L_B \mid B \in G(\C,U)\}$. The product on $\Atp(\C,U)$ is given by
\begin{eqnarray*}
    (L^{-1}_B,U^{-1}_B,L_B) \cdot (L^{-1}_{B'},U^{-1}_{B'},L_{B'}) \hskip -2cm &&\\
    &=&
    (L_B,U_B,L^{-1}_B)^\rho (L^{-1}_{B'},U^{-1}_{B'},L_{B'}) (L_B,U_B,L^{-1}_B)^{\rho^2} \\
    &=& (U^S_B,L^{-1}_B,L^S_B)(L^{-1}_{B'},U^{-1}_{B'},L_{B'})(L^{-S}_B,L^S_B,U_B)\\
    &=& (U^{-1}_B,L^{-1}_B,R^{-1}_B)(L^{-1}_{B'},U^{-1}_{B'},L_{B'})(R_B,R^{-1}_B,U_B)\\
    &=& (U^{-1}_B*L^{-1}_{B'}*R_B,
    L^{-1}_B*U^{-1}_{B'}*R^{-1}_B,R^{-1}_B*L_{B'}*U_B).
\end{eqnarray*}
By the middle Moufang-Hopf identity, the last component $R^{-1}_B*L_{B'}*U_B$ of this triple acts by
\begin{displaymath}
x(R^{-1}_B*L_{B'}*U_B)_c = \sum ((1B_\ci)(1B'_\ci))x
\end{displaymath}
so, under the identification of $\M(\Atp(\C,U))$ with $\{ L_B \mid B \in G(\C,U)\}$,  the product is given by
\begin{displaymath}
L_B \cdot L_{B'} = L_{L_B*R_{B'}}.
\end{displaymath}

\begin{proposition}
 $\M(\Atp(\C,U))$ and $\morco(\C,U)$ are isomorphic Moufang loops.
\end{proposition}
\begin{proof}
Given a coalgebra morphism $\theta \colon \C \to U$, define $L_\theta\colon \C \to \Endo(U)$ by $L_\theta \colon c \mapsto L_{c\theta}$. Since $L_\theta = L_\theta'$ if and only if $\theta = \theta'$, we can identify $\morco(\C,U)$ with $\{ L_\theta \mid \theta \in \morco(\C,U)\}$.

The elements $L_\theta$ with $\theta \in \morco(\C,U)$ belong to $G(\C,U)$, and they also satisfy $L_{L_\theta} = L_\theta$ since $(L_{L_\theta})_c = L_{1(L_\theta)_c} = L_{c\theta} = (L_\theta)_c$. Hence, $\{ L_\theta \mid \theta \in \morco(\C,U)\} \subseteq \{ L_B \mid B \in G(\C,U)\}$. The other inclusion also holds. Given $B \in G(\C,U)$, define $\theta$ by $c \theta = 1B_c$. Then $L_\theta = L_B$. The product on $\{ L_\theta \mid \theta \in \morco(\C,U)\}$  is
\begin{eqnarray*}
x(L_\theta \cdot L_{\theta'})_c &=& x(L_{L_\theta} \cdot L_{L_{\theta'}})_c = \sum ((1(L_\theta)_\ci)(1(L_{\theta'})_\cii))x = \sum (\theta_\ci \theta'_\cii)x\\
&=& x(L_{\theta * \theta'})_c.
\end{eqnarray*}
Consequently the Moufang loop $\morco(\C,U)$ is isomorphic to $\M(\Atp_\C(U))$.
\end{proof}

% % ----------------------------------------------------------------
% \bibliographystyle{amsplain}
% \bibliography{Hopf_algebras_with_triality}
% %%%%% End Added

\end{document}